\documentclass[11pt]{amsart}
\usepackage[margin=1.2in]{geometry}
\usepackage[maxbibnames=99, style=alphabetic, maxalphanames=99, minalphanames=6, 
sorting=nyt, giveninits=true]{biblatex}
\AtEveryBibitem{
    \clearfield{issn}
    \clearfield{isbn}
    \clearfield{doi}
    \clearlist{location}
    \clearfield{month}
    \clearfield{day}
    \clearfield{series}
    \clearlist{language}
    \clearlist{translator}
    \renewbibmacro{in:}{}
}
\DeclareLabelalphaNameTemplate{
  \namepart[use=true, pre=true, strwidth=1, compound=true]{prefix}
  \namepart[compound=true]{family}
}

\usepackage{amssymb, mathtools, amsthm, thmtools}
\usepackage{mathdots}
\usepackage{graphicx}
\usepackage{microtype}
\usepackage{accents}
\usepackage{enumitem}
\usepackage{indentfirst}
\usepackage{xcolor}
\usepackage{tikz-cd}
\usepackage{csquotes}
\usepackage[colorlinks=true, linkcolor=blue]{hyperref} 
\usepackage[noabbrev]{cleveref}
\usepackage{bookmark}

\makeatletter
\ifdefined\newcounteralias 
  \let\fixthm@newcounteralias\newcounteralias
  \renewcommand*\newcounteralias[2]{%
    \def\fixthm@a{#1}\def\fixthm@b{#2}%
    \ifx\fixthm@a\fixthm@b\else\fixthm@newcounteralias{#1}{#2}\fi}
\fi
\makeatother

\theoremstyle{definition}
\declaretheorem[name=Definition,refname={definition,definitions},Refname={Definition,Definitions},numberwithin=section]{definition}

\theoremstyle{plain}
\declaretheorem[name=Lemma,refname={lemma,lemmas},Refname={Lemma,Lemmas},sibling=definition]{lemma}
\declaretheorem[name=Proposition,refname={proposition,propositions},Refname={Proposition,Propositions},sibling=definition]{prop}
\declaretheorem[name=Theorem,refname={theorem,theorems},Refname={Theorem,Theorems},sibling=definition]{thm}
\declaretheorem[name=Corollary,refname={corollary,corollaries},Refname={Corollary,Corollaries},sibling=definition]{corollary}
\theoremstyle{remark}
\declaretheorem[name=Remark,refname={remark,remarks},Refname={Remark,Remarks},numbered=no]{rmk}

\newcommand{\GL}{\mathrm{GL}}
\newcommand{\numberset}{\mathbb}

\newcommand{\C}{\numberset{C}}

\newcommand{\Z}{\numberset{Z}}
\newcommand{\N}{\numberset{N}}
\newcommand{\G}{\numberset{G}}
\newcommand{\git}{/\!\!/}

\newcommand{\coulomb}[2]{\mathcal{M}_C({#1},{#2})} 
\DeclareMathOperator{\id}{Id}
\DeclareMathOperator{\im}{im}

\newcommand{\gl}{\mathfrak{gl}}

\newcommand{\Wh}{\mathrm{Wh}}
\newcommand{\reg}{\mathrm{reg}}

\DeclareMathOperator{\Hom}{Hom}
\DeclareMathOperator{\spec}{Spec}

\DeclareMathOperator{\diag}{diag}
\DeclareMathOperator{\Ad}{Ad}
\DeclareMathOperator{\Mat}{Mat}
\DeclareMathOperator{\End}{End}

\author{Bruno da Silveira Dias}
\title{Rankin-Selberg duality via gluing}
\address{Department of Mathematics, The University of Chicago}
\email{brunodias@uchicago.edu}
\date{}

\begin{document}	
    \begin{abstract}
        We use a gluing procedure introduced by Ginzburg to describe the relative Langlands dual of the hyperspherical Hamiltonian $(\GL_n \times \GL_m)$-variety $T^*(\Hom(\C^m,\C^n))$, and in particular the Rankin-Selberg case $m=n$. We show that the dual is isomorphic to the triangle part of Cherkis-Nakajima-Takayama bow varieties, recovering a result of Nakajima. Using this result, we give a new proof of a theorem of Braverman-Finkelberg-Nakajima describing the regular perverse sheaf in type $A$. Following a suggestion of Ginzburg, we explain how to modify the gluing so that the dual Hamiltonian variety of $T^*\mathbf{N}$, for any gluable finite-dimensional representation $\mathbf{N}$ of a complex reductive group $G$, is naturally equipped with an anti-symplectic involution, and give an explicit formula for this involution in the Rankin-Selberg case.
    \end{abstract}
    \maketitle
    
    \section{Introduction}\label{section-intro}
        
        Let $G = \GL_n \times \GL_m$ act on $\mathbf{N} = \Hom(\C^m,\C^n)$ by $(g,h) \cdot a = gah^{-1}$. The cotangent bundle $T^*\mathbf{N}$ is a hyperspherical Hamiltonian $G$-variety in the sense of \cite{benzvi2024relativelanglandsduality}. It is also one of the basic building blocks, along with the triangle part recalled below, used in the quiver description of Cherkis-Nakajima-Takayama bow varieties \cite{NakajimaTakayama2017}, which are built from these by Hamiltonian reduction. In this paper, we use a gluing procedure introduced in \cite{ginzburg2025pointwisepurityderivedsatake} to describe the relative Langlands dual of $T^*\mathbf{N}$, defined by means of ring objects in the derived Satake category as in \cite[Section 5]{BFN-Ring}, \cite[Section 8]{benzvi2024relativelanglandsduality}.

        Recall the definition of the \emph{triangle part}, introduced by Takayama in \cite{Takayama2016}. We fix a dimension vector $(m,n) \in \N^2$, and let $\mathcal{M} := \mathcal{M}_{(m,n)}$ be the set of quintuples of linear maps $(x,y_1,y_2,v,u)$ as pictured in the diagram
        
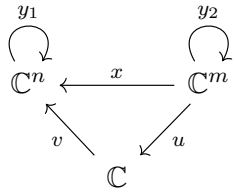
\[
            \begin{tikzcd}[column sep = small]
                \C^n  \arrow[out=120,in=60,loop,looseness=5,"y_1"] & & \C^m \arrow[ll, "x"']\arrow[out=120,in=60,loop,looseness=5,"y_2"] \arrow[ld, "u"] \\
                & \C \arrow[lu, "v"]& 
            \end{tikzcd}
        \]
        satisfying the equation
        \begin{align*}
            & y_1x-xy_2+vu=0,
        \end{align*}
        and the stability conditions:
        \begin{itemize}
            \item There is no $y_1$-invariant proper subspace $S \subseteq \C^n$ such that $S \supseteq \im(x)+\im(v)$;
            \item There is no $y_2$-invariant non-zero subspace $S \subseteq \C^m$ such that $S \subseteq \ker(x) \cap\ker(u)$.
        \end{itemize}
        As shown by Nakajima and Takayama in \cite[Section 3.1]{NakajimaTakayama2017}, based on results by Hurtubise \cite{Hurtubise1989}, Bielawski \cite{Bielawski1997}, and Takayama \cite{Takayama2016}, $\mathcal{M}$ is a smooth affine symplectic algebraic variety. We view $\mathcal{M}$ as a hyperspherical Hamiltonian $(\GL_n \times \GL_m)$-variety via the action
        \begin{equation}
            (g,h) \cdot (x,y_1,y_2,v,u) = (gxh^T,gy_1g^{-1},(h^{T})^{-1}y_2h^T,gv,uh^{T}), \label{eq-correctaction}
        \end{equation}
        where $(-)^T$ denotes matrix transposition. This differs from the action considered in \cite{NakajimaTakayama2017} by an outer automorphism.
        
        Our main result is the following: 
        
        \begin{thm}\label{thm-main}
            The relative Langlands dual of the hyperspherical $(\GL_n \times \GL_m)$-variety $T^*(\Hom(\C^m,\C^n))$ is $\mathcal{M}_{(m,n)}$, viewed as a $(\GL_n \times \GL_m)$-variety via the action \eqref{eq-correctaction}.
        \end{thm}

        When $m=n$, the triangle part $\mathcal{M}$ is isomorphic to $T^*(\GL_n \times \C^n)$ viewed as a hyperspherical Hamiltonian $(\GL_n \times \GL_n)$-variety via the cotangent lift of the action on $\GL_n \times \C^n$ by $(g,h) \cdot (a,v) = (gah^T,gv)$. \Cref{thm-main} is then the \emph{Rankin-Selberg} case of relative Langlands duality: the dual of $G = \GL_n \times \GL_n \curvearrowright T^*\gl_n$ is $G^\vee = \GL_n \times \GL_n \curvearrowright T^*(\GL_n \times \C^n)$. For $m<n$, $\mathcal{M}_{(m,n)}$ is isomorphic to a \emph{hyperspherical equivariant slice} $\GL_n \times \mathcal{S}_e$ corresponding to a nilpotent element $e \in \gl_n$ of Jordan type $(n-m, 1^m)$, as in \cite{Finkelberg_2025}, and analogously with the roles of $m$ and $n$ reversed if $n<m$.

        It has been observed by Nakajima in \cite{nakajima2024sdualhamiltonianmathbfg} that \Cref{thm-main} follows from the identification of the Coulomb branch of quiver gauge theories of affine type $A$ with bow varieties, the main result of \cite{NakajimaTakayama2017}, combined with \cite[Theorem 2.11]{BFN-Ring}. Our proof of \Cref{thm-main} relies instead on a direct gluing argument introduced by Ginzburg in \cite{ginzburg2025pointwisepurityderivedsatake}, building on work of Teleman \cite{Teleman2021} and Gannon-Webster \cite{gannon2025functorialitycoulombbranches}. 

        We use \Cref{thm-main}, together with results from \cite{ginzburg2025pointwisepurityderivedsatake} and \cite{gannon2025functorialitycoulombbranches} to give a new proof of \cite[Theorem 2.11]{BFN-Ring}, showing that the regular perverse sheaf in type $A$ is isomorphic to a ring object associated to the Gelfand-Tsetlin quiver.

        This paper is organized as follows: in \Cref{section-prelim}, we review the gluing procedure of Ginzburg and compute the Teleman section of the representation $\Hom(\C^m,\C^n)$. In \Cref{section-bezout}, we briefly discuss Bezoutian matrices and derive some identities that play an important role in our proof of the main theorem, which is given in \Cref{section-mainthm}. In \Cref{section-gt} we compute the Coulomb branch of associated to the Gelfand-Tsetlin quiver, and give a new proof of \cite[Theorem 2.11]{BFN-Ring}. In \Cref{section-involution}, following an approach suggested by Ginzburg, we explain how the gluing can be modified so that the resulting Hamiltonian space is naturally equipped with an anti-symplectic involution. We then give a formula for this involution in the Rankin-Selberg case $m=n$.

        \subsection*{Acknowledgments}
        
        I am deeply grateful to my advisor, Victor Ginzburg, for suggesting this problem to me, for the many helpful discussions, and for reading and making several useful suggestions on preliminary drafts. I also thank Michael Finkelberg, Tom Gannon, and Tsao-Hsien Chen for providing interesting comments on an early draft. I am thankful to Yiannis Sakellaridis for detailed and enlightening observations.

        \subsection*{Computational Resource Disclosure} ChatGPT 6 was used for proofreading and to simplify the proof of \Cref{lemma-coulombbranchmain}. All writing was done by the author.

    \section{Preliminaries}\label{section-prelim}
        
        Let $G$ be a complex reductive group and $G^\vee$ its Langlands dual. We write $T$, resp.\ $T^\vee$, for a maximal torus of $G$, resp.\ $G^\vee$, $\mathfrak{g} = \mathrm{Lie}(G)$, resp. $\check{\mathfrak{g}} = \mathrm{Lie}(G^\vee)$, $\mathfrak{t} := \mathrm{Lie}(T)$, resp.\ $\check{\mathfrak{t}} := \mathrm{Lie}(T^\vee)$, and $W$ for the Weyl group. Let $\mathfrak{c}_G := \check{\mathfrak{g}}^* \git G^\vee \cong \check{\mathfrak{t}}^*/W$ be the coadjoint quotient, and $J_{G^\vee} \to \mathfrak{c}_G$ the group scheme of universal centralizers.

        We write $\mathrm{Gr}_G$ for the affine Grassmannian of $G$ and, for a finite-dimensional representation $\mathbf{N}$ of $G$, we let $\mathcal{A}_{G,\mathbf{N}}$ denote the corresponding ring object in the ind-completed derived Satake category $D_{G(O)}(\mathrm{Gr}_G)$, as defined by Braverman-Finkelberg-Nakajima in \cite{BFN2018,BFN-Ring}. The Coulomb branch $\mathcal{M}_C(G,\mathbf{N})$ is the affine scheme
        \[
            \mathcal{M}_C(G,\mathbf{N}) := \spec(H_{G(O)}^\bullet(\mathrm{Gr}_G,\mathcal{A}_{G,\mathbf{N}})).
        \]
        It is a Poisson scheme over $\mathfrak{c}_G$, and is equipped with an action of $\mathcal{M}_C(G,0) \cong J_{G^\vee}$, \cite{BFN2018,BFM,Teleman2021,gannon2025functorialitycoulombbranches}.

        Let $\mathcal{R}_G \in D_{G(O)}(\mathrm{Gr}_G)$ denote the regular perverse sheaf -- the object corresponding to the regular representation $\C[G^\vee]$ under the geometric Satake equivalence. By \cite[Section 5]{BFN-Ring}, cf.\ also \cite[Section 5]{ginzburg2025pointwisepurityderivedsatake}, the derived Satake equivalence $\Phi$ may be defined as the functor $\mathcal{F}\mapsto R\Gamma_{G(O)}(\mathrm{Gr}_G, \mathcal{R}_G \otimes^! \mathcal{F})$, and if $\mathcal{F}$ is a ring object this identification is compatible with the multiplications. For the BFN ring object $\mathcal{A}_{G,\mathbf{N}}$, $H^\bullet(\Phi(\mathcal{A}_{G,\mathbf{N}})) \cong H_{G(O)}^\bullet(\mathrm{Gr}_G, \mathcal{R}_G \otimes^! \mathcal{A}_{G,\mathbf{N}})$ is a commutative algebra over $\mathrm{Sym}(\check{\mathfrak{g}}[-2])$; see \cite{BezrukavnikovFinkelberg2008,nakajima2024sdualhamiltonianmathbfg}. In this setting, the relative Langlands dual of $T^*\mathbf{N}$ is defined as
        \[
            (T^*\mathbf{N})^\vee:= \spec H_{G(O)}^\bullet(\mathrm{Gr}_G, \mathcal{R}_G \otimes^! \mathcal{A}_{G,\mathbf{N}}).
        \]

        \subsection{Relative Langlands dual via gluing}\label{subsection-gluingoutline}
        
        To a representation $\rho\colon G \to \GL(\mathbf{N})$, Teleman \cite{Teleman2021} associated a $W$-equivariant rational map $\epsilon_\rho\colon \check{\mathfrak{t}}^* \dashrightarrow T^\vee$. Explicitly, this is given as follows (see also \cite[Section 2.1]{gannon2025functorialitycoulombbranches}): let $\lambda_1,\dots,\lambda_d$ be the non-zero weights of $\rho$, viewed as cocharacters $\G_m \to T^\vee$. For each $i$, we write $\mathrm{d}\lambda_i$ for the linear function on $\mathfrak{t}$ defined by the differential of $\lambda_i$. Then each non-zero weight $\lambda_i$ gives rise to a rational map $\mathfrak{t} = \check{\mathfrak{t}}^* \dashrightarrow T^\vee$, $\xi \mapsto \lambda_i(\mathrm{d}\lambda_i(\xi))$, and $\epsilon_\rho$ is defined as their product:
        \begin{equation}\label{eq-telemansection}
            \epsilon_\rho\colon \check{\mathfrak{t}}^* \dashrightarrow T^\vee,\quad \xi\mapsto \prod_{i=1}^d \lambda_i(\mathrm{d}\lambda_i(\xi))
        \end{equation}
        By the functor of points description of the universal centralizer (\cite{Lusztig1976}, \cite{Knop1996}, \cite{DonagiGaitsgoryTheGerbeofHiggsBundles}, \cite{BFM}, \cite{NgoLeLemmeFondamentalPourLesAlgebresdeLie}), this $W$-invariant rational map corresponds to a rational section of the structure morphism $\pi\colon J_{G^\vee} \to \mathfrak{c}_G$, which we also denote by $\epsilon_\rho$. This yields a birational automorphism $\sigma_\rho\colon J_{G^\vee} \dashrightarrow J_{G^\vee}$ by means of the multiplication in the group scheme $J_{G^\vee}$, i.e.\ $j \mapsto \epsilon_\rho(\pi(j)) \cdot j$.
        
        In their paper \cite{gannon2025functorialitycoulombbranches}, Gannon and Webster introduced a \emph{gluability} condition on the representation $\rho$ and proved that when $\rho$ is gluable, the coordinate ring of the corresponding Coulomb branch of cotangent type $\coulomb{G}{\mathbf{N}}$ can be described as
        \[
            \mathcal{O}(\coulomb{G}{\mathbf{N}}) \cong \{f \in \mathcal{O}(J_{G^\vee}) \mid \sigma_\rho^*(f) \in \mathcal{O}(J_{G^\vee})\}.
        \]
        In other words, $\coulomb{G}{\mathbf{N}}$ is the affinization of the scheme obtained by gluing two copies of the universal centralizer $J_{G^\vee}$ using $\sigma_\rho$ as transition function. According to the paragraph following the statement of Theorem 2.8 in \cite{gannon2025functorialitycoulombbranches}, it is a sufficient condition for gluability of $\rho$ that its image in $\GL(\mathbf{N})$ contains the set of scalar matrices. This will suffice for our purposes, as one can easily check that the representations we consider satisfy this requirement.

        In \cite{ginzburg2025pointwisepurityderivedsatake}, Ginzburg extended this gluing construction to describe the relative Langlands dual of $T^*\mathbf{N}$, as follows. Let $\check{\mathfrak{g}}^*_{\mathrm{reg}}$ denote the regular locus of $\check{\mathfrak{g}}^*$ and let
        \[
            I_{\mathrm{reg}}:=\{(g,\xi) \in G^\vee \times \check{\mathfrak{g}}^*_{\mathrm{reg}} \mid \Ad g(\xi) = \xi\}.
        \]                
        The second projection gives $I_\mathrm{reg}$ the structure of a group scheme over $\check{\mathfrak{g}}^*_{\mathrm{reg}}$, canonically isomorphic to the pullback of $J_{G^\vee}$ by the coadjoint quotient $\chi\colon \check{\mathfrak{g}}^*_{\mathrm{reg}} \to \mathfrak{c}_G$. Pulling back the Teleman section $\epsilon_\rho$ then yields a rational section $\check{\mathfrak{g}}^*_{\mathrm{reg}} \dashrightarrow I_{\mathrm{reg}}$ of the second projection, of the form $\xi \mapsto (\gamma_\rho(\xi),\xi)$ for a rational map $\gamma_\rho\colon \check{\mathfrak{g}}^*_{\mathrm{reg}} \dashrightarrow G^\vee$. 
        
        Fix $U^\vee \subseteq G^\vee$ a maximal unipotent group and a non-degenerate character $\psi\colon \check{\mathfrak{u}}=\mathrm{Lie}(U^\vee) \to \C$. Let $\mathrm{Wh}_{G^\vee} := T^*_{\psi}(G^\vee/U^\vee) = G^\vee \times_{U^\vee} (\psi + \check{\mathfrak{u}}^\perp)$ denote the Whittaker-twisted cotangent bundle of $G^\vee/U^\vee$. It is a smooth affine symplectic variety equipped with a Hamiltonian action of $G^\vee$ with moment map $\mu\colon \mathrm{Wh}_{G^\vee} \to \check{\mathfrak{g}}^*_{\mathrm{reg}}$. The action of the rational map $\gamma_\rho$ gives a $G^\vee$-equivariant birational automorphism
        \begin{align}
            \sigma_\rho \colon \mathrm{Wh}_{G^\vee} &\dashrightarrow \mathrm{Wh}_{G^\vee} \nonumber \\
            x &\mapsto \gamma_{\rho}(\mu(x))\cdot x. \label{eq-automorphism}
        \end{align}
        By \cite[Theorem 3.2.4]{ginzburg2025pointwisepurityderivedsatake}, if the representation $\rho$ is gluable, then the coordinate ring of the relative Langlands dual of $T^*\mathbf{N}$ is
        \begin{equation}\label{eq-algebrasigma}
            \mathcal{A}_{\sigma} := \{f \in \mathcal{O}(\mathrm{Wh}_{G^\vee}) \mid \sigma_{\rho}^*(f) \in \mathcal{O}(\mathrm{Wh}_{G^\vee})\}.
        \end{equation}
        Geometrically, this means that the relative Langlands dual of $T^*\mathbf{N}$ can be described as the affinization of the scheme obtained by gluing two copies of the Whittaker-twisted cotangent bundle $\mathrm{Wh}_{G^\vee}$, one of them shifted by the action of $\gamma_\rho$.

        \subsection{Setup and notation}\label{subsection-setup}
        
        We let $G = G^\vee = \GL_n \times \GL_m$ and $\rho\colon G \to \GL(\Hom(\C^m,\C^n))$ be the representation given by $(g,h) \cdot a = gah^{-1}$. Our goal is to use the gluing procedure outlined above to describe the relative Langlands dual of the $G$-variety $T^*\mathbf{N}$, where $\mathbf{N} := \Hom(\C^m, \C^n)$ is viewed as a spherical $G$-variety via the representation $\rho$. Here, spherical means the variety has a dense open Borel orbit. For an irreducible affine variety $X$, this condition is equivalent to the coordinate ring $\mathcal{O}(X)$ being a multiplicity-free $G$-module (see \cite{Vinberg1978}, \cite{Kac1980}). It is known that the cotangent bundle to a spherical variety $X$ is \emph{coisotropic} in the sense that the field of invariant rational functions $\C(T^*X)^G$ is Poisson commutative, \cite{Vinberg2001}. When some additional conditions are satisfied, such coisotropic varieties are called \emph{hyperspherical}; see \cite[Section 3.5]{benzvi2024relativelanglandsduality}.
        
        Let $T$ be the maximal torus of diagonal matrices in $G$, and denote by $W$ the Weyl group. We identify $\gl_n$, resp.\ $\gl_m$, with its dual via the trace pairing. We write $\mathfrak{c}_n$, resp.\ $\mathfrak{c}_m$, for the coadjoint quotient $\gl_n \git \GL_n$, resp.\ $\gl_m \git \GL_m$, identified with the affine space of monic polynomials of degree $n$, resp.\ $m$.
        
        Recall that the \emph{companion matrix} for a monic polynomial $p = t^n + p_{n-1}t^{n-1} + \dots + p_0$ is the matrix
        \[
            C_p=
            \begin{pmatrix*}
                0 & 0 & \dots & 0 & -p_0 \\
                1 & 0 & \dots & 0 & -p_1 \\
                0 & 1 & \dots & 0 & -p_2 \\
                \vdots & \vdots & \ddots & \vdots & \vdots \\
                0 & 0 & \dots & 1 & -p_{n-1}
            \end{pmatrix*}
        \]
        representing the action of $t$ in $\C[t]/(p(t))$ with respect to the basis $1,t,\dots,t^{n-1}$. The matrix $C_p$ is regular, having the first standard basis vector $e_1 = (1,0,\dots,0)$ as a cyclic vector, and its characteristic polynomial is $p$. The assignment $p \mapsto C_p$ is a section of the (co)adjoint quotient $\chi\colon \gl_{n} \to \mathfrak{c}_n$. Denoting by $\mathcal{S}_n \subseteq \gl_n$ the slice of companion matrices and using the left-invariant trivialization $T^*\GL_n \cong \GL_n \times \gl_n$, we obtain an isomorphism of $\GL_n$-varieties $\mathrm{Wh}_{\GL_n} \cong \GL_n \times \mathcal{S}_n$, the action on the latter being given by $g \cdot (g',a) = (gg',a)$.

        Thus we may (and will) identify the Whittaker-twisted cotangent bundle $\mathrm{Wh}_{G^\vee}$ of $G^\vee = \GL_n \times \GL_m$ with the variety $\GL_n \times \GL_m \times \mathcal{S}_n \times \mathcal{S}_m$. Under this identification, the $G^\vee$-action on $\mathrm{Wh}_{G^\vee}$ is given by $(g,h) \cdot (g',h',a,b) = (gg',hh',a,b)$ and the moment map is
        \begin{align}
            \mu\colon \mathrm{Wh}_{G^\vee} = \GL_n \times \GL_m \times \mathcal{S}_n \times \mathcal{S}_m &\to \gl_n \times \gl_m \nonumber \\
            (g,h,a,b) &\mapsto (gag^{-1},hbh^{-1}). \label{eq-momentmap}
        \end{align}        

        \subsubsection{The Teleman section}
    
        The weights $\lambda_{ij}$ of the representation $\rho\colon G \to \GL(\Hom(\C^m,\C^n))$ are such that
        \[
            \mathrm{d}\lambda_{ij}\colon \diag(z_1,\dots,z_n,w_1,\dots,w_m) \mapsto z_i-w_j
        \]
        for each $(i,j) \in \{1,\dots,n\} \times \{1,\dots,m\}$. From formula \eqref{eq-telemansection}, it follows that the corresponding $W$-invariant rational map $\check{\mathfrak{t}}^* \dashrightarrow T^\vee$ takes $(a,b) := (\diag(z_1,\dots,z_n),\diag(w_1,\dots,w_m))$ to
        \[
            \textstyle{ \left( \diag\left(\prod_{j=1}^m(z_1-w_j),\dots,\prod_{j=1}^m(z_n-w_j)\right), \diag \Big(\prod_{i=1}^n(z_i-w_1)^{-1},\dots,\prod_{i=1}^n(z_i-w_m)^{-1} \Big)\right)}
        \]
        Since the characteristic polynomial of $a = \diag(z_1,\dots,z_n)$, resp.\ $b = \diag(w_1,\dots,w_m)$, is $\chi_a(t) = \prod_{j=1}^n(t-z_j)$, resp.\ $\chi_b(t) = \prod_{j=1}^m(t-w_j)$, the displayed formula above is the same as $(\chi_b(a),(-1)^n\chi_a(b)^{-1})$. Thus, since the characteristic polynomial is conjugation-invariant, it follows that the rational map $\gamma_\rho \colon \gl_n \times \gl_m \dashrightarrow \GL_n \times \GL_m$ is given by
        \begin{equation}\label{eq-gluing-section}
            (a,b) \mapsto (\chi_{b}(a),(-1)^n\chi_{a}(b)^{-1}).
        \end{equation}

        \subsubsection{The gluing}

        Combining \eqref{eq-gluing-section} and \eqref{eq-momentmap}, the birational automorphism $\sigma_{\rho}\colon \mathrm{Wh}_{G^\vee} \dashrightarrow \mathrm{Wh}_{G^\vee}$ defined as in \eqref{eq-automorphism} takes the form:
        \begin{align}
            \sigma_{\rho}(g,h,a,b) &= (\chi_{hbh^{-1}}(gag^{-1})g, (-1)^n\chi_{gag^{-1}}(hbh^{-1})^{-1} h,a,b) \nonumber \\
            &= (\chi_b(gag^{-1})g, (-1)^n\chi_a(hbh^{-1})^{-1}h, a, b) \nonumber \\
            &= (g\chi_b(a),(-1)^nh\chi_a(b)^{-1},a,b) \nonumber
        \end{align}
        Since $a = C_p$ and $b = C_q$ for unique monic polynomials $p \in \mathfrak{c}_n$, $q \in \mathfrak{c}_m$, the above can be rewritten as
        \begin{equation}\label{eq-formula-sigma}
            \sigma_{\rho}(g,h,C_p,C_q) = (gq(C_p), (-1)^nhp(C_q)^{-1}, C_p, C_q).
        \end{equation}

        We then define an algebra
        \[
            \mathcal{A}_{\sigma} := \{f \in \mathcal{O}(\mathrm{Wh}_{G^\vee}) \mid \sigma_{\rho}^*(f) \in \mathcal{O}(\mathrm{Wh}_{G^\vee})\}.
        \]
        As explained above, by \cite[Theorem 3.2.4]{ginzburg2025pointwisepurityderivedsatake}, proving \Cref{thm-main} amounts to proving the following:
        \begin{thm}\label{thm-main2}
            There is an isomorphism of $G^\vee$-schemes $\spec(\mathcal{A}_{\sigma}) \cong \mathcal{M}_{(m,n)}$.
        \end{thm}
        This proof is given in \Cref{section-mainthm}. The strategy of the proof is to construct two open embeddings $\Phi_1,\Phi_2\colon \mathrm{Wh}_{G^\vee} \hookrightarrow \mathcal{M}_{(m,n)}$ satisfying $\Phi_2 = \Phi_1 \circ \sigma_\rho$, and use algebraic Hartogs' lemma to obtain the desired isomorphism. To construct these embeddings, we use some results about companion and Bezoutian matrices, discussed below in \Cref{section-bezout}.

    \section{Bezoutians}\label{section-bezout}
        
        Given a monic polynomial $p \in \C[t]$ of degree $n$, let $V_p := \C[t]/(p(t))$. As we observed above, with respect to the basis $1,t,\dots,t^{n-1}$ of $V_p$, the action of $t$ is represented by the companion matrix $C_p$. Furthermore, $V_p$ is equipped with a non-degenerate bilinear form $\mathrm{Tr}\colon V_p \times V_p \to \C$ that makes it into a \emph{Frobenius algebra}. This form is given by
        \[
            (f,g) \mapsto -\mathrm{Res}_{\infty}\left(\frac{f(t)g(t)}{p(t)}\mathrm{d}t\right)
        \]
        where $\mathrm{Res}_{\infty}$ denotes the residue at infinity.
        
        Concretely, this means that the dual space $V_p^* = \Hom(V_p,\C)$ is a $V_p$-module, and the linear function $\eta_p\colon V_p \to \C$ that extracts the $t^{n-1}$-coefficient is a cyclic generator of $V_p^*$. It follows from this that there exists a unique isomorphism of $V_p$-modules $S_p\colon V_p^* \to V_p$ such that $S_p(\eta_p) = 1$. If $p(t) = t^{n} + p_{n-1}t^{n-1} + \dots + p_0$, then with respect to the basis $1,t,\dots,t^{n-1}$ and the dual basis of $V_p^*$, one finds that $S_p$ is given by the matrix
        \begin{equation}\label{eq-Sp-matrix}
            S_p =
            \begin{pmatrix}
                p_1 & p_2 & \dots & p_{n-1} & 1 \\
                p_2 & \dots & \dots & 1 & 0 \\
                \vdots & \vdots & \dots & \vdots & \vdots \\
                p_{n-1} & 1 & \dots & 0 & 0 \\
                1 & 0 & \dots & 0 & 0
            \end{pmatrix}
        \end{equation}

        We have an isomorphism $V_p \otimes V_p \cong \Hom(V_p^*,V_p)$ given by
        \begin{align*}
            \sum f_i \otimes g_i \longmapsto \left(
            \varphi \mapsto \sum \varphi(g_i)f_i\right).
        \end{align*}
        The image of $\xi \in V_p \otimes V_p$ is a $V_p$-module homomorphism if and only if $(t \otimes 1 - 1 \otimes t)\xi = 0$. Identifying $V_p \otimes V_p \cong \C[x,y]/(p(x),p(y))$ so that $x = t \otimes 1$ and $y = 1 \otimes t$, we see that any polynomial $\Delta \in \C[x,y]$ such that $(x-y)\Delta \in (p(x),p(y))$ defines a $V_p$-module homomorphism $V_p^* \to V_p$. For any $f \in \C[t]$, the rational function
        \begin{equation}\label{eq-bez_def}
            \Delta_{p,f} := \frac{p(x)f(y)-p(y)f(x)}{x-y}
        \end{equation}
        is a polynomial such that $(x-y)\Delta_{p,f} \in (p(x),p(y))$, since interchanging $x$ and $y$ flips the sign of the numerator in the fraction above. It thus gives rise to a $V_p$-module homomorphism $B_{p,f}\colon V_p^* \to V_p$. The matrix of $B_{p,f}$ is formed by the coefficients of the residue class of $\Delta_{p,f}$ in $V_p \otimes V_p$ with respect to the basis $x^iy^j$, $0 \leq i,j < n$. When $\deg(f) \leq n$, these entries are the coefficients of the polynomial $\Delta_{p,f}$ itself:
        \begin{equation}\label{eq-bez_matrix_def}
            B_{p,f} = (b_{ij})_{i,j=1,\dots,n}, \text{ where } \Delta_{p,f} = \sum_{i,j=1}^{n}b_{ij}x^{i-1}y^{j-1}.
        \end{equation}
        This recovers the classical Bezoutian matrix of the polynomials $p$ and $f$, as defined by Sylvester in 1853; see \cite{Sylvester1853}.
        
        Note that $B_{p,f}^T=B_{p,f}$, since the corresponding element of $V_p \otimes V_p$ is symmetric. Moreover, the fact that $B_{p,f}$ is a $V_p$-module homomorphism translates into the matrix equation
        \[
            C_pB_{p,f} = B_{p,f}C_p^T.
        \]
        When $f=1$, we get a homomorphism $B_{p,1}\colon V_p^* \to V_p$ defined by
        \[
            \Delta_{p,1} = \frac{p(x) - p(y)}{x-y}.
        \]
        By direct computation, one sees that $B_{p,1}(\eta_p) = 1$, and so $B_{p,1}$ is the isomorphism $S_p$.

        \begin{rmk}
            By \cite{HeinigRost}, an invertible $n \times n$ matrix $B$ is a Bezoutian if and only if $B^{-1}$ is a Hankel matrix.
        \end{rmk}

        \subsection{Some Useful Identities}

        We now derive some matrix identities that are crucial for the construction of the open embeddings used in the proof of our main result. To this end, let $p,q \in \C[t]$ be monic, of degrees $n$ and $m$ respectively, and $C_p, C_q$ their companion matrices. In what follows, we consider the space $\Mat_{n \times m}(\C)$ of $n \times m$ matrices as a module over $\C[x,y]$, with $x$ acting as $X \mapsto C_pX$, and $y$ as $X \mapsto XC_q^T$.
        
        Let $e_1,\dots,e_n$, resp.\ $f_1,\dots,f_m$, denote the standard basis of $\C^n$, resp.\ of $\C^m$, thought of as column vectors. We define matrices $K_{p,q}, L_{p,q} \in \Mat_{n\times m}$ by
        \begin{align*}
            K_{p,q} &:= \Delta_{p,1}(x,y) \cdot e_1f_1^T \\
            L_{p,q} &:= \Delta_{q,1}(x,y) \cdot e_1f_1^T
        \end{align*}
        
        \begin{lemma}\label{lemma-bezouttriangles}
            The following matrix identities hold:
            \begin{align}
                C_pK_{p,q}-K_{p,q}C_q^T &= -e_1f_1^Tp(C_q^T)  \label{eq_rankone1} \\
                C_pL_{p,q} - L_{p,q}C_q^T &= q(C_p)e_1f_1^T  \label{eq_rankone2} \\
                q(C_p)K_{p,q} &= -L_{p,q}p(C_q^T) \label{eq_useful}
            \end{align}
        \end{lemma}
        \begin{proof}
            We have
            \begin{align*}
                C_pK_{p,q}-K_{p,q}C_q^T &= (x-y)\Delta_{p,1}(x,y) \cdot e_1f_1^T \\
                &= (p(x)-p(y)) \cdot e_1f_1^T \\
                &= p(C_p)e_1f_1^T-e_1f_1^Tp(C_q^T) \\
                &= - e_1f_1^Tp(C_q^T)
            \end{align*}
            proving \eqref{eq_rankone1}. The proof of \eqref{eq_rankone2} is similar.

            Next, we note the following polynomial equalities:
            \begin{align*}
                \Delta_{p,q}(x,y) = q(x)\Delta_{p,1}(x,y)-p(x)\Delta_{q,1}(x,y) = q(y)\Delta_{p,1}(x,y)-p(y)\Delta_{q,1}(x,y)
            \end{align*}
            Applying $\Delta_{p,q}(x,y)$ to the rank-one matrix $e_1f_1^T$ then yields
            \[
                q(C_p)K_{p,q}-p(C_p)L_{p,q} = K_{p,q}q(C_q^T)-L_{p,q}p(C_q^T).
            \]
            This proves \eqref{eq_useful}, since $p(C_p) = 0$ and $q(C_q^T) = 0$.
        \end{proof}

        \begin{lemma}\label{lemma-factorization}
            The $n \times m$ matrices $K_{p,q}$ and $L_{p,q}$ can be factored as follows:
            \begin{itemize}
                \item if $m \leq n$, then 
                \[
                    K_{p,q} = S_p \begin{pmatrix*}
                        \id_{m \times m} \\ *
                    \end{pmatrix*}, \qquad 
                    L_{p,q} = \begin{pmatrix*}
                        \id_{m \times m} \\ 0
                    \end{pmatrix*} S_q,
                \]
                where $*$ denotes a possibly non-zero $(n-m) \times m$ block, and $0$ denotes a zero $(n-m) \times m$ block.

                \item if $m > n$, then 
                \[
                    K_{p,q} = S_p \begin{pmatrix*}
                        \id_{n \times n} & 0
                    \end{pmatrix*}, \qquad 
                    L_{p,q} = \begin{pmatrix*}
                        \id_{n \times n} & *
                    \end{pmatrix*} S_q,
                \]
                where $*$ denotes a possibly non-zero $n \times (m-n)$ block, and $0$ denotes a zero $n \times (m-n)$ block.
            \end{itemize}
        \end{lemma}
        \begin{proof}
            This can be seen by expanding the polynomials $\Delta_{p,1}$ and $\Delta_{q,1}$, and using the explicit action of companion matrices on basis vectors. Alternatively, one can argue as follows. Define a linear map $K\colon V_p \to V_q$ given on the basis $t^j \pmod{p}$ by 
            \[
                t^j \pmod{p} \mapsto t^j \pmod{q}.
            \]
            Then $K_{p,q}$ is the matrix of the composition $S_p \circ K^T\colon V_q^* \to V_p$, which gives the desired factorization. A similar argument handles the statements about $L_{p,q}$.
        \end{proof}
        \begin{corollary}\label{corollary-bez}
            Both $K_{p,q}$ and $L_{p,q}$ have full rank and they satisfy $e_n^TK_{p,q} = f_1^T$, $L_{p,q}f_m=e_1$. 
        \end{corollary}
        \begin{proof}
            This follows from the factorizations in \Cref{lemma-factorization}, together with the fact that the Bezoutian matrices $S_p$ and $S_q$ are invertible and satisfy $e_n^TS_p = e_1^T$, $S_qf_m=f_1$ (see \eqref{eq-Sp-matrix}).
        \end{proof}

    \section{Main result}\label{section-mainthm}

        \subsection{Proof of \texorpdfstring{\Cref{thm-main2}}{Theorem 2.1}}

        We recall the definition of the \emph{triangle part}, introduced by Takayama in \cite{Takayama2016}. Fix a dimension vector $(m,n) \in \N^2$ and let $\mathcal{M} := \mathcal{M}_{(m,n)}$ be the set of quintuples of linear maps $(x,y_1,y_2,v,u)$ as pictured in the diagram below\footnote{In \cite{NakajimaTakayama2017}, elements of a triangle part are denoted $(A,B_1,B_2,a,b)$, with $A \in \Hom(\C^m,\C^n)$, $B_1 \in \End(\C^m)$, $B_2 \in \End(\C^n)$, $a \in \C^n$, $b \in (\C^m)^*$. The correspondence with our notation is $x=A,\ y_1 = B_2,\ y_2 = B_1,\ v = a,\ u=b$. In particular, the indices $1$ and $2$ are interchanged.}
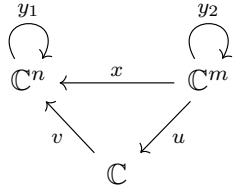
        \[
            \begin{tikzcd}[column sep = small]
                \C^n  \arrow[out=120,in=60,loop,looseness=5,"y_1"] & & \C^m \arrow[ll, "x"']\arrow[out=120,in=60,loop,looseness=5,"y_2"] \arrow[ld, "u"] \\
                & \C \arrow[lu, "v"]& 
            \end{tikzcd}
        \]
        satisfying the equation
        \begin{align}
            & y_1x-xy_2+vu=0 \label{eq-condition0}
        \end{align}
        and the stability conditions
        \begin{align}
            & \text{There is no $y_1$-invariant proper subspace $S \subseteq \C^n$ such that $S \supseteq \im(x)+\im(v)$}, \label{eq-condition1} \\
            & \text{There is no $y_2$-invariant non-zero subspace $S \subseteq \C^m$ such that $S \subseteq \ker(x) \cap\ker(u)$}. \label{eq-condition2}
        \end{align}
        It follows from \cite[Prop. 2.9]{Takayama2016} that $\mathcal{M}$ is a smooth symplectic affine algebraic variety of dimension $m^2+n^2 +m +n$. As explained in \cite[Section 3.1]{NakajimaTakayama2017}, if $m=n$ then $\mathcal{M}$ is isomorphic to $T^*(\GL_n \times \C^n)$, and if $m<n$ then $\mathcal{M}$ is isomorphic to $\GL_n \times \mathcal{S}_e$, where $\mathcal{S}_e \subseteq \gl_n$ is a Slodowy slice corresponding to a nilpotent element $e$ of Jordan type $(n-m,1^m)$.
        
        There is a Hamiltonian $(\GL_n \times \GL_m)$-action on $\mathcal{M}$ given by 
        \begin{equation}\label{eq-actionontriangles}
            (g,h) \cdot (x,y_1,y_2,v,u) = (gxh^T,gy_1g^{-1},(h^{T})^{-1}y_2h^T,gv,uh^{T}).
        \end{equation}
        One checks that the corresponding moment map $\mu_{\mathcal{M}}\colon \mathcal{M} \to \gl_n \times \gl_m$ is given by the formula $(x,y_1,y_2,v,u) \mapsto (y_1,y_2^T)$. We note that this action differs from the one considered in \cite{NakajimaTakayama2017} by an outer automorphism.
    
        Recall from \Cref{section-prelim} that $G = G^\vee = \GL_n \times \GL_m$, and $\rho\colon G \to \GL(\Hom(\C^m,\C^n))$ is the representation given by $(g,h) \cdot a = gah^{-1}$. Also, $\mathfrak{c}_n = \gl_n \git \GL_n$, resp.\ $\mathfrak{c}_m = \gl_m \git \GL_m$, are the spaces of monic polynomials of degree $n$, resp.\ $m$, and $\mathcal{S}_n \subseteq \gl_n$, resp.\ $\mathcal{S}_m \subseteq \gl_m$ denote the corresponding slices of companion matrices, writing $C_p$ for the companion matrix of a monic polynomial $p$. We use the identification $\mathrm{Wh}_{G^\vee} = \GL_n \times \GL_m \times \mathcal{S}_n \times \mathcal{S}_m$, so that elements in the Whittaker-twisted cotangent bundle $\mathrm{Wh}_{G^\vee}$ are quadruples $(g,h,C_p,C_q)$ with $g \in \GL_n, h \in \GL_m, p \in \mathfrak{c}_n, q \in \mathfrak{c}_m$. 
        
        We computed the Teleman section and the corresponding birational automorphism $\sigma_\rho$ of $\mathrm{Wh}_{G^\vee}$, given by \eqref{eq-formula-sigma}:
        \[
            \sigma_{\rho}(g,h,C_p,C_q) = (gq(C_p), (-1)^nhp(C_q)^{-1}, C_p, C_q).
        \]
        Following \Cref{subsection-setup}, we define an algebra
        \[
            \mathcal{A}_{\sigma} := \{f \in \mathcal{O}(\mathrm{Wh}_{G^\vee}) \mid \sigma_{\rho}^*(f) \in \mathcal{O}(\mathrm{Wh}_{G^\vee})\}.
        \]
        Our goal is to prove \Cref{thm-main2}: there is an isomorphism of $G^\vee$-schemes
        \begin{equation}\label{eq-mainthmiso}
            \spec(\mathcal{A}_\sigma) \cong \mathcal{M}_{(m,n)}.
        \end{equation}
        To prove this isomorphism, we consider the following subvarieties of $\mathcal{M}$:
        \begin{align}
            \mathcal{M}_1 &:= \{ (x,y_1,y_2,v,u) \in \mathcal{M} \mid v \text{ is cyclic for } y_1 \}, \label{eq-open1}\\
            \mathcal{M}_2 &:= \{ (x,y_1,y_2,v,u) \in \mathcal{M} \mid u \text{ is cyclic for } y_2^T \}. \label{eq-open2}
        \end{align}
        These are both $G^\vee$-stable and open, their respective complements in $\mathcal{M}$ being the divisors given by the equations
        \begin{align*}
           \det(v \wedge y_1v \wedge y_1^2v \wedge \dots \wedge y_1^{n-1}v) = 0, \\
           \det(u \wedge uy_2 \wedge uy_2^2 \wedge \dots \wedge uy_2^{m-1}) = 0.
        \end{align*}

        \begin{prop}\label{prop-main}
            There exist two $G^\vee$-equivariant open embeddings $\Phi_1,\Phi_2 \colon \mathrm{Wh}_{G^\vee} \hookrightarrow \mathcal{M}$ whose images are $\mathcal{M}_1$ and $\mathcal{M}_2$, respectively, and such that $\Phi_2 = \Phi_1 \circ \sigma_{\rho}$.
        \end{prop}
        
        \begin{proof}[Proof that \Cref{prop-main} implies \eqref{eq-mainthmiso}]
            Since $\Phi_2 = \Phi_1 \circ \sigma_{\rho}$, we have an algebra homomorphism
            \[
                \Phi_1^*\colon \mathcal{O}(\mathcal{M}) \to \mathcal{A}_\sigma \subseteq \mathcal{O}(\mathrm{Wh}_{G^\vee}),\quad f \mapsto f \circ \Phi_1.
            \]
            It is injective, as $\Phi_1$ is an open embedding and $\mathcal{M}$ is irreducible.

            Since $\mathcal{M}$ is normal and the complement of $\mathcal{M}_1\cup\mathcal{M}_2$ is an intersection of distinct divisors, it follows from algebraic Hartogs' lemma that any function that is regular on both $\mathcal{M}_1$ and $\mathcal{M}_2$ is in fact regular on $\mathcal{M}$. Hence the algebra homomorphism above is surjective, and so it induces an isomorphism of affine schemes $\spec(\mathcal{A}_{\sigma}) \cong \mathcal{M}$.
        \end{proof}

        \subsubsection*{Construction of the maps $\Phi_1\colon \mathrm{Wh}_{G^\vee} \to \mathcal{M}_1$ and $\Phi_2\colon \mathrm{Wh}_{G^\vee} \to \mathcal{M}_2$}
        
        Let $e_1,\dots,e_n$ denote the standard basis of $\C^n$, $f_1,\dots,f_m$ that of $\C^m$. In what follows we identify elements of $\C^n$ and $\C^m$ with column vectors, and elements of the dual spaces $(\C^n)^*$ and $(\C^m)^*$ with row vectors.
        
        Given $p \in \mathfrak{c}_n$ and $q \in \mathfrak{c}_m$, we have a diagram
        \[
            \begin{tikzcd}[column sep = small]
                \C^n  \arrow[out=120,in=60,loop,looseness=5,"C_p"] & & \C^m \arrow[ll, "K_{p,q}"']\arrow[out=120,in=60,loop,looseness=5,"C_q^T"] \arrow[ld, "f_1^Tp(C_q^T)"] \\
                & \C \arrow[lu, "e_1"]& 
            \end{tikzcd}
        \]
        where $K_{p,q}$ was defined in \Cref{section-bezout}. By \eqref{eq_rankone1}, the linear maps in this diagram satisfy equation \eqref{eq-condition0}. The stability condition \eqref{eq-condition1} holds because $e_1$ is a cyclic vector for $C_p$. To see that condition \eqref{eq-condition2} also holds, we use the fact that $e_n^TK_{p,q}=f_1^T$ (\Cref{corollary-bez}). Thus, if $S \subseteq \ker(K_{p,q})$ is $C_q^T$-invariant and $\xi \in S$, then $f_1^T(C_q^T)^j\xi = e_n^TK_{p,q}(C_q^T)^{j}\xi = 0$ for all $j \geq 0$. Since $f_1^T$ is cyclic for the action of $C_q^T$ on the right, this implies $\xi = 0$ and hence $S = 0$.
        
        Therefore, the above diagram gives a point in $\mathcal{M}_1$. We let $\Phi_1\colon \mathrm{Wh}_{G^\vee} \to \mathcal{M}_1$ be given by
        \begin{equation}\label{eq-phi1}
            \Phi_1(g,h,C_p,C_q) = (gK_{p,q}h^T, gC_pg^{-1}, (h^T)^{-1}C_q^Th^T,ge_1, f_1^Tp(C_q^T)h^T).
        \end{equation}

        Similarly, there is a diagram
        \[
            \begin{tikzcd}[column sep = small]
                \C^n  \arrow[out=120,in=60,loop,looseness=5,"C_p"] & & \C^m \arrow[ll, "-L_{p,q}"']\arrow[out=120,in=60,loop,looseness=5,"C_q^T"] \arrow[ld, "f_1^T"] \\
                & \C \arrow[lu, "q(C_p)e_1"]& 
            \end{tikzcd}
        \]
        where $L_{p,q}$ was defined in \Cref{section-bezout}. Equation \eqref{eq_rankone2}, the fact that $f_1^T$ is cyclic for the action of $C_q^T$ on the right, and the fact that $L_{p,q}f_m=e_1$ is cyclic for $C_p$ guarantee that the stability conditions hold. We then let $\Phi_2\colon \mathrm{Wh}_{G^\vee} \to \mathcal{M}_2$ be the map
        \begin{equation}\label{eq-phi2}
            \Phi_2(g,h,C_p,C_q) = ((-1)^{n-1}gL_{p,q}h^T, gC_pg^{-1}, (h^T)^{-1}C_q^Th^T, gq(C_p)e_1,(-1)^nf_1^Th^T).
        \end{equation}

        \begin{lemma}
            The maps $\Phi_1\colon \mathrm{Wh}_{G^\vee} \to \mathcal{M}_1$ and $\Phi_2\colon \mathrm{Wh}_{G^\vee} \to \mathcal{M}_2$ defined by \eqref{eq-phi1} and \eqref{eq-phi2} are isomorphisms satisfying $\Phi_2 = \Phi_1 \circ \sigma_{\rho}$.
        \end{lemma}
        \begin{proof}
            The equality $\Phi_2 = \Phi_1 \circ \sigma_{\rho}$ follows immediately from the definition of the maps and the identity in \eqref{eq_useful}. We prove that $\Phi_1$ is an isomorphism by constructing an explicit inverse. The proof for $\Phi_2$ is analogous.

            Given $(x,y_1,y_2,v,u) \in \mathcal{M}_1$, let $p:=\chi(y_1)$ and $q := \chi(y_2)$ be the characteristic polynomials of $y_1$ and $y_2$, respectively. Since $v$ is a cyclic vector for $y_1$ and $e_1$ is a cyclic vector for the companion matrix $C_p$, there exists a unique $g \in \GL_n$ such that $gC_pg^{-1} = y_1$ and $ge_1 = v$. Consider the covector $\xi:=e_n^Tg^{-1}x \in (\C^m)^*$. We claim that $u = \xi p(y_2)$. To see this, observe first that
            \begin{equation*}
                C_pg^{-1}x-g^{-1}xy_2 = g^{-1}(y_1x-xy_2) = -g^{-1}vu = -e_1u.
            \end{equation*}
            Thus $g^{-1}xy_2 = C_pg^{-1}x+e_1u$, and by induction on $k$ it follows that
            \[
                g^{-1}xy_2^k = C_p^kg^{-1}x+\sum_{j=1}^{k}C_p^{j-1}e_1uy_2^{k-j} = C_p^kg^{-1}x+\sum_{j=1}^{k}e_{j}uy_2^{k-j}
            \]
            for each $k \in \{1,\dots,n\}$, where we have used that $C_p^{j-1}e_1 = e_{j}$ for all $j \in \{1,\dots,n\}$. Multiplying by $e_n^T$ on the left, we obtain
            \[
                \xi y_2^k = e_n^TC_p^kg^{-1}x+\delta_{nk}u.
            \]
            By $\C$-linearity and since $p$ is monic, it follows that $\xi p(y_2) = e_n^Tp(C_p)g^{-1}x + u = u$, proving the claim.

            Now, consider the $y_2$-invariant subspace $S := \bigcap_{k \geq 0}\ker(\xi y_2^k) \subseteq \C^m$. It is clearly contained in $\ker(x)$. Moreover, since $u = \xi p(y_2)$ it follows that $S \subseteq \ker(u)$. The stability condition \eqref{eq-condition2} then implies $S = 0$; hence, $\xi$ is a cyclic covector for the action of $y_2$ on the right. Thus, there exists a unique $h \in \GL_m$ such that $(h^T)^{-1}C_q^Th^T = y_2$ and $f_1^Th^T = \xi$. 
            
            Altogether, we have defined a map $\Psi_1\colon \mathcal{M}_1 \to \mathrm{Wh}_{G^\vee},\ (x,y_1,y_2,v,u) \mapsto (g,h,C_p,C_q)$. It remains to check that this map is inverse to $\Phi_1$. Verifying $\Phi_1 \circ \Psi_1 = \id_{\mathcal{M}_1}$ boils down to checking the following equalities:
            \begin{itemize}
                \item $g^{-1}x = K_{p,q}h^T$
                \item $u = f_1^Tp(C_q^T)h^T$
            \end{itemize}
            To prove the first equality, let $x' = g^{-1}x - K_{p,q}h^T$ and observe that $C_px'=x'y_2$. The definitions of $\xi$ and $h$ imply that
            \[
                e_n^Tx'=e_n^Tg^{-1}x-e_n^TK_{p,q}h^T = \xi - f_1^Th^T = 0.
            \]
            Hence $e_n^TC_p^kx'=e_n^Tx'y_2^k=0$ for every $k$, which implies $g^{-1}x=K_{p,q}h^T$ as $e_n^T$ is cyclic for the action of $C_p$ on the right. The second identity follows by combining the first with $u = \xi p(y_2)$, $\xi = e_n^Tg^{-1}x$ and $p(y_2) = (h^T)^{-1}p(C_q^T)h^T$ to get $u=e_n^TK_{p,q}p(C_q^T)h^T = f_1^Tp(C_q^T)h^T$.

            The verification that $\Psi_1 \circ \Phi_1 = \id_{{\mathrm{Wh}}_{G^\vee}}$ is straightforward, so we omit it.
        \end{proof}

        \subsection{The Rankin-Selberg case}

        In the case $m=n$, Takayama's triangle part $\mathcal{M}_{(m,n)}$ is isomorphic as a symplectic algebraic variety to the cotangent bundle $T^*(\GL_n \times \C^n)$. Using the left-invariant trivialization to write $T^*(\GL_n \times \C^n) = \GL_n \times \gl_n \times \C^n \times (\C^n)^*$, this isomorphism is given explicitly by
        \begin{align*}
            \mathcal{M}_{(n,n)} &\to T^*(\GL_n \times \C^n) \\
            (x,y_1,y_2,v,u) &\mapsto (x,y_2,v,-ux^{-1}).
        \end{align*}
        We note that the above is well-defined since the stability conditions imply that $x$ has full rank \cite[Lemma 2.18]{Takayama2016}, and thus invertible in the case $m=n$. The inverse is    
        \begin{align*}
            T^*(\GL_n \times \C^n) &\to \mathcal{M}_{(n,n)} \\
            (a,y,v,u) &\mapsto (a, aya^{-1}+vu,y,v,-ua).
        \end{align*}
        Under this isomorphism, the open subsets $\mathcal{M}_1$ and $\mathcal{M}_2$ defined in \eqref{eq-open1} and \eqref{eq-open2} are identified with the following open subsets:
        \begin{align*}
            \mathcal{M}_1 &:= \{ (a,y,v,u) \in T^*(\GL_n \times \C^n) \mid v \text{ is cyclic for } aya^{-1} \}, \\
            \mathcal{M}_2 &:= \{ (a,y,v,u) \in T^*(\GL_n \times \C^n) \mid u \text{ is cyclic for } (aya^{-1})^T \}.
        \end{align*}
        Moreover, we have $K_{p,q} = S_p$ and $L_{p,q} = S_q$, so the open embeddings $\Phi_1$ and $\Phi_2$ are given by the formulas
        \begin{align}
            \Phi_1(g,h,C_p,C_q) &= (gS_ph^T,(h^T)^{-1}C_q^Th^T,ge_1,e_n^Tq(C_p)g^{-1}) \label{eq-morphism1-RS} \\
            \Phi_2(g,h,C_p,C_q) &= ((-1)^{n-1}gS_q h^T,(h^T)^{-1}C_q^Th^T,gq(C_p)e_1,e_n^Tg^{-1}). \label{eq-morphism2-RS}
        \end{align}

        \begin{corollary}\label{corollary-RS}
            The relative Langlands dual of the $(\GL_n \times \GL_n)$-variety $T^*(\gl_n)$ is the variety $T^*(\GL_n \times \C^n) = \GL_n \times \gl_n \times \C^n \times (\C^n)^*$, equipped with the $(\GL_n \times \GL_n)$-action given by
            \[
                (g,h) \cdot (a,y,v,u) = (gah^T, (h^T)^{-1}yh^T, gv, ug^{-1}).
            \]
        \end{corollary}

        \begin{rmk}
            When $m=n$, the equality $\Phi_2 = \Phi_1 \circ \sigma_\rho$ involves the matrix identity $q(C_p)S_p=-S_qp(C_q^T)$, which is a special case of \eqref{eq_useful}. This can also be seen as a consequence of the factorization $B_{p,q} = q(C_p)S_p$ and the fact that $B_{p,q} = -B_{q,p}$ for monic polynomials $p$ and $q$ of degree $n$, both of which are classical facts about Bezoutian matrices; see e.g.\ \cite{HELMKE19891039}.
        \end{rmk}

            \subsubsection{Example: the case \texorpdfstring{$n=1$}{n=1}}
            For $n=1$, we have $G = G^\vee = \C^\times \times \C^\times$, and the rational map $\gamma_\rho\colon \C \times \C \dashrightarrow \C^\times \times \C^\times$ is given by $(a,b) \mapsto (a-b,(a-b)^{-1})$. Moreover, $\mathrm{Wh}_{G^\vee} = \C^\times \times \C^\times \times \C \times \C$, with $G^\vee$-action $(g,h) \cdot (g',h',a,b) = (gg', hh', a, b)$ and moment map $\mu(g,h,a,b) = (a,b)$. Thus the rational automorphism $\sigma_{\rho}\colon \mathrm{Wh}_{G^\vee} \dashrightarrow \mathrm{Wh}_{G^\vee}$ is given by
            \[
                (g,h,a,b) \mapsto ((a-b)g,(a-b)^{-1}h,a,b).
            \] 
            The algebra $\mathcal{A}_{\sigma}$ can be computed directly, as follows. Since $\mathcal{O}(\mathrm{Wh}_{G^\vee}) = \C[g,h,a,b,g^{-1},h^{-1}]$, any $f \in \mathcal{O}(\mathrm{Wh}_{G^\vee})$ can be written as $f = \sum_{i,j \in \Z} f_{ij}(a,b)g^{i}h^{j}$ for some polynomials $f_{ij}$ in two variables. Then $f \in \mathcal{A}_{\sigma}$ if and only if the pullback
            \[
                \sigma_\rho^*(f) = \sum_{i,j \in \Z} f_{ij}(a,b)(a-b)^{i-j}g^ih^j
            \]
            is a regular function on $\mathrm{Wh}_{G^\vee}$, which in turn holds if and only if $f_{ij}(a,b)(a-b)^{i-j} \in \C[a,b]$ for every $i,j \in \Z$. It follows that $\mathcal{A}_{\sigma}$ is the subalgebra of $\mathcal{O}(\mathrm{Wh}_{G^\vee})$ generated by the elements $a,b,g,h^{-1},gh,h^{-1}g^{-1}, (a-b)g^{-1},(a-b)h$.
    
            Setting
            \[
                (gh,b,g,(a-b)g^{-1}) =: (z,y,v,u)
            \]
            yields the isomorphism $\spec(\mathcal{A}_{\sigma}) \cong \C^\times \times \C \times \C \times \C$, the $G^\vee$-action on the right-hand side being $(g,h) \cdot (z,y,v,u) = (gzh,y,gv,ug^{-1})$.

    \section{The Gelfand-Tsetlin case}\label{section-gt}

        In this section, we give a new proof of \cite[Theorem 2.11]{BFN-Ring}, which establishes an isomorphism between the regular perverse sheaf on the affine Grassmannian of type A and a certain ring object associated to the Gelfand-Tsetlin quiver.
        
        Fix an integer $n \geq 1$. Let $\mathcal{G}_n := \prod_{i=1}^{n}\GL_i$ and consider the Gelfand-Tsetlin quiver:
        \[
            \begin{tikzpicture}[x=1.2cm,
                gauge/.style={circle, draw, inner sep=1pt, minimum size=6pt, label distance=1.5pt},
                every label/.append style={text height=1.5ex, text depth=0.25ex}]
                \node[gauge, label=below:{$n$}]   (1) at (0,0) {};
                \node[gauge, label=below:{$n-1$}] (2) at (1,0) {};
                \node[gauge, label=below:{$n-2$}] (3) at (2,0) {};
                \node                             (4) at (3,0) {$\cdots$};
                \node[gauge, label=below:{$2$}]   (5) at (4,0) {};
                \node[gauge, label=below:{$1$}]   (6) at (5,0) {};
                \draw[->] (6)--(5);
                \draw[->] (5)--(4);
                \draw[->] (4)--(3);
                \draw[->] (3)--(2);
                \draw[->] (2)--(1);
            \end{tikzpicture}
        \]
        Let $\mathcal{V}_n:=\bigoplus_{i=1}^{n-1}\Hom(\C^i,\C^{i+1})$ be the corresponding representation of $\mathcal{G}_n$.
        \begin{prop}\label{prop-maincoulombbranch}
            The Coulomb branch $\mathcal{M}_{C}(\mathcal{G}_n,\mathcal{V}_n)$ is isomorphic as a Poisson scheme over $\mathfrak{c}_{\GL_n}$ to the Whittaker-twisted cotangent bundle $\mathrm{Wh}_{\GL_n}$.
        \end{prop}
    
        We prove \Cref{prop-maincoulombbranch} by induction, using the idea of quiver dismemberment of \cite[Section 1.2]{gannon2025functorialitycoulombbranches}, as follows. Assume $n>2$ and consider the disconnected quiver obtained from the Gelfand-Tsetlin quiver by ``carving'' the $(n-1)$ vertex into two:
        \[
            \begin{tikzpicture}[x=1.2cm,
              gauge/.style={circle, draw, inner sep=1pt, minimum size=6pt, label distance=1.5pt},
              every label/.append style={text height=1.5ex, text depth=0.25ex}]
              \node[gauge, label=below:{$n$}]   (a1) at (0,0) {};
              \node[gauge, label=below:{$n-1$}] (a2) at (1,0) {};
              \draw[->] (a2)--(a1);
              \node[gauge, label=below:{$n-1$}] (b1) at (2.5,0) {};
              \node[gauge, label=below:{$n-2$}] (b2) at (3.5,0) {};
              \node                             (b3) at (4.5,0) {$\cdots$};
              \node[gauge, label=below:{$2$}]   (b4) at (5.5,0) {};
              \node[gauge, label=below:{$1$}]   (b5) at (6.5,0) {};
              \draw[->] (b5)--(b4);
              \draw[->] (b4)--(b3);
              \draw[->] (b3)--(b2);
              \draw[->] (b2)--(b1);
            \end{tikzpicture}
        \]
        Let $\mathcal{H}_n:= (\GL_n \times \GL_{n-1}) \times \mathcal{G}_{n-1}$, and write $\widetilde{\mathcal{V}}_n:= \Hom(\C^{n-1},\C^n) \oplus \mathcal{V}_{n-1}$ for the representation of $\mathcal{H}_n$ afforded by the dismembered quiver above. Let $\zeta\colon \mathcal{G}_n \to \mathcal{H}_n$ be the homomorphism given by
        \[
            (g_n,\dots,g_1)\mapsto ((g_n,g_{n-1}),(g_{n-1},\dots,g_1)).
        \]
        By functoriality of Coulomb branches \cite[Corollary 1.5]{gannon2025functorialitycoulombbranches}, there is an isomorphism of affine schemes
        \begin{equation}\label{eq-dismemberment1}
            \mathcal{M}_C(\mathcal{G}_n,\mathcal{V}_n) \cong \mathrm{Aff}(J_{\mathcal{G}_n} \times_{\mathfrak{c}_{\mathcal{H}_n}}^{J_{\mathcal{H}_n}} \mathcal{M}_C(\mathcal{H}_n,\widetilde{\mathcal{V}}_n)),
        \end{equation}
        where $\mathrm{Aff}$ stands for affinization and $\times^{J_{\mathcal{H}_n}}$ denotes the quotient by the anti-diagonal action of $J_{\mathcal{H}_n}$.
    
        Now, by induction,
        \[
            \mathcal{M}_C(\mathcal{H}_n,\widetilde{\mathcal{V}}_n) \cong \mathcal{M}_C(\GL_{n} \times \GL_{n-1},\Hom(\C^{n-1},\C^n)) \times \Wh_{\GL_{n-1}}.
        \]
        Moreover, since the homomorphism $\zeta$ is the identity in every component except $\GL_{n-1}$, on which it is a diagonal embedding, formula \eqref{eq-dismemberment1} simplifies to
        \[
            \mathcal{M}_C(\mathcal{G}_n,\mathcal{V}_n) \cong \mathrm{Aff}(\mathcal{M}_C(\GL_{n} \times \GL_{n-1},\Hom(\C^{n-1},\C^n)) \times_{\mathfrak{c}_{\GL_{n-1}}}^{J_{\GL_{n-1}}}\Wh_{\GL_{n-1}}),
        \]
        where $J_{\GL_{n-1}}$ acts anti-diagonally.
    
        \begin{lemma}\label{lemma-coulombbranchmain}
            For $n \ge 2$, there is a $J_{\GL_n}$-equivariant isomorphism
            \[
                \mathrm{Aff}(\mathcal{M}_C(\GL_{n} \times \GL_{n-1},\Hom(\C^{n-1},\C^n)) \times_{\mathfrak{c}_{\GL_{n-1}}}^{J_{\GL_{n-1}}}\Wh_{\GL_{n-1}}) \cong \Wh_{\GL_n}.
            \]
        \end{lemma}
    
        \begin{rmk}
            The $J_{\GL_{n-1}}$-action on $\Wh_{\GL_n}$ is as follows. Consider the inclusion $\iota\colon\GL_{n-1} \hookrightarrow \GL_n$ as the upper-left block. The map of Lie algebras $\gl_{n-1} \hookrightarrow \gl_n$ induces a map $\mathrm{pr}\colon \gl_n^* \twoheadrightarrow \gl_{n-1}^*$ that under the trace-form identification is given by taking a matrix $\xi$ to its principal $(n-1) \times (n-1)$ block. Via the inclusion $\iota$, the variety $\Wh_{\GL_n}$ is equipped with a Hamiltonian action of $\GL_{n-1}$, and the corresponding moment map is
            \begin{align}
                \nu\colon \Wh_{\GL_n} = \GL_n \times \mathcal{S}_n &\to \gl_{n-1} \label{eq-momentmap-n-1} \\
                (a, C_p) &\mapsto \mathrm{pr}(aC_pa^{-1}). \nonumber
            \end{align}
            The universal centralizer $J_{\GL_{n-1}}$ acts on $\Wh_{\GL_{n}}$ by
            \begin{equation}\label{eq-jaction1}
                (f,q) \cdot (a, C_p) = (\iota(f(\nu(a,C_p)))a, C_p) = (\iota(f(\mathrm{pr}(aC_pa^{-1})))a,C_p),
            \end{equation}
            where $q \in \mathfrak{c}_{\GL_{n-1}}$ is the characteristic polynomial of $\nu(a,C_p)$, and $f \in (\C[t]/(q))^\times$ is a polynomial of degree $<n-1$ invertible modulo $q$. The action of $J_{\GL_{n-1}}$ commutes with the usual action of $J_{\GL_n}$, so $\Wh_{\GL_n}$ is a scheme over $\mathfrak{c}_{\GL_n \times \GL_{n-1}}$ equipped with an action of $J_{\GL_n \times \GL_{n-1}}$.
        \end{rmk}
        
        \subsubsection*{Description of $\mathcal{M}_C(\GL_{n} \times \GL_{n-1},\Hom(\C^{n-1},\C^n))$}
        
        In \Cref{thm-main} we proved that the relative Langlands dual of $\GL_n \times \GL_{m} \curvearrowright T^*(\Hom(\C^{m},\C^n))$ is the triangle part of Nakajima and Takayama. In the special case $m=n-1$, this variety is isomorphic to $T^*\GL_n = \GL_n \times \gl_n$ equipped with the $\GL_n \times \GL_{n-1}$ action given by
        \begin{equation}\label{eq-reldualaction}
            (g,h) \cdot (a,\xi) = (gah^T, (h^T)^{-1}\xi h^T),
        \end{equation}
        where we identify $h \in \GL_{n-1}$ with its image in $\GL_n$ under the inclusion $\iota\colon \GL_{n-1} \hookrightarrow \GL_n$. The moment map for this action is
        \begin{align}
            \mu\colon \GL_n \times \gl_n &\to \gl_n \times \gl_{n-1} \nonumber\\
            (a, \xi) &\mapsto (a\xi a^{-1}, \mathrm{pr}(\xi)^T). \label{eq-momentmap1edge}
        \end{align}    
        By \cite{BezrukavnikovFinkelberg2008}, the corresponding Coulomb branch $\mathcal{M}_C(\GL_{n} \times \GL_{n-1},\Hom(\C^{n-1},\C^n))$ is therefore isomorphic to $\mu^{-1}(\mathcal{S}_n \times \mathcal{S}_{n-1})$. Concretely, this is the set of $(a,\xi)\in\GL_n\times\gl_n$ such that there exist (unique) $p \in \mathfrak{c}_{\GL_n}$ and $q \in \mathfrak{c}_{\GL_{n-1}}$ such that $a\xi a^{-1} = C_p$ and $\mathrm{pr}(\xi) = C_q^T$. The same description of this Coulomb branch is obtained as a special case of \cite[Lemma 3.7]{gannon2025functorialitycoulombbranches}.
    
        The action of $J_{\GL_n \times \GL_{n-1}}$ on $\mathcal{M}_C(\GL_{n} \times \GL_{n-1},\Hom(\C^{n-1},\C^n)) \cong \mu^{-1}(\mathcal{S}_n \times \mathcal{S}_{n-1})$ is as follows. Let $(a,\xi) \in \mu^{-1}(\mathcal{S}_n \times \mathcal{S}_{n-1})$ and, as above, let $p \in \mathfrak{c}_{\GL_n}$ and $q \in \mathfrak{c}_{\GL_{n-1}}$ be the unique elements such that $a\xi a^{-1} = C_p$ and $\mathrm{pr}(\xi) = C_q^T$. Let $(j,p) \in J_{\GL_n}$ and $(f,q) \in J_{\GL_{n-1}}$, where $j \in (\C[t]/(p))^\times$ and $f \in (\C[t]/(q))^{\times}$. Then
        \begin{align}
            (j,p)\cdot(a,\xi) &= (j(C_p)a,\xi) \label{eq-jaction1edge1} \\
            (f,q)\cdot(a,\xi) &= (a\iota(f(C_q^T)), \iota(f(C_q^T))^{-1}\xi\iota(f(C_q^T))) \label{eq-jaction1edge2}.
        \end{align}
    
        \begin{proof}[Proof of \Cref{lemma-coulombbranchmain}]
            We use the above description of $\mathcal{M}_C(\GL_{n} \times \GL_{n-1},\Hom(\C^{n-1},\C^n))$ as the variety $\mu^{-1}(\mathcal{S}_n \times \mathcal{S}_{n-1}) \subseteq \GL_n \times \gl_n$. Given $(a,\xi) \in \mu^{-1}(\mathcal{S}_n \times \mathcal{S}_{n-1})$, there are unique $p \in \mathfrak{c}_{\GL_n}$ and $q \in \mathfrak{c}_{\GL_{n-1}}$ such that $a\xi a^{-1} = C_p$ and $\mathrm{pr}(\xi) = C_q^T$. Thus, a $\C$-point of the fiber product
            \[
                \mu^{-1}(\mathcal{S}_n \times \mathcal{S}_{n-1}) \times_{\mathfrak{c}_{\GL_{n-1}}} \Wh_{\GL_{n-1}}
            \]
            corresponds to $((a,\xi), (b,C_q)) \in \mu^{-1}(\mathcal{S}_n \times \mathcal{S}_{n-1}) \times (\GL_{n-1} \times \mathcal{S}_{n-1})$ such that $\mathrm{pr}(\xi)=C_q^T$, with $b\in\GL_{n-1}$ arbitrary.
            
            Let $S_p$ be the Bezoutian matrix defined in \eqref{eq-bez_matrix_def}. Define a map
            \begin{align*}
                \Phi\colon \mu^{-1}(\mathcal{S}_n \times \mathcal{S}_{n-1}) \times_{\mathfrak{c}_{\GL_{n-1}}} \Wh_{\GL_{n-1}} &\to \Wh_{\GL_n} \\
                ((a,\xi), (b,C_q)) &\mapsto (\iota(b)a^TS_p^{-1}, a\xi a^{-1}).
            \end{align*}
            Let $(f,q)\in J_{\GL_{n-1}}$ and put $K:=f(C_q)$, so that $f(C_q^T)=K^T$. The anti-diagonal action of $J_{\GL_{n-1}}$ on the fiber product is
            \[
                (f,q)\cdot((a,\xi),(b,C_q)) = ((a\iota(K^T),\iota(K^T)^{-1}\xi\iota(K^T)),(bK^{-1},C_q)).
            \]
            Hence
            \begin{align*}
                \Phi((f,q)\cdot((a,\xi),(b,C_q))) &= (\iota(bK^{-1})(a\iota(K^T))^TS_p^{-1}, C_p) \\
                &= (\iota(bK^{-1})\iota(K)a^TS_p^{-1}, C_p) \\
                &= \Phi((a,\xi),(b,C_q)),
            \end{align*}
            and so $\Phi$ induces a morphism on the quotient by the anti-diagonal action:
            \[
                \Phi\colon \mu^{-1}(\mathcal{S}_n\times\mathcal{S}_{n-1})\times_{\mathfrak{c}_{\GL_{n-1}}}^{J_{\GL_{n-1}}} \Wh_{\GL_{n-1}} \to \Wh_{\GL_n}.
            \]
            If $(g,C_p)=\Phi((a,\xi),(b,C_q))$, then
            \begin{align*}
                gC_pg^{-1} &= \iota(b)a^TS_p^{-1}C_pS_p(a^T)^{-1}\iota(b^{-1}) \\
                &= \iota(b)a^TC_p^T(a^T)^{-1}\iota(b^{-1}) \\
                &= \iota(b)\xi^T\iota(b^{-1}).
            \end{align*}
            Therefore, for the moment map $\nu$ of \eqref{eq-momentmap-n-1},
            \[
                \nu(g,C_p) = \mathrm{pr}(gC_pg^{-1}) = b\,\mathrm{pr}(\xi)^T b^{-1} = bC_qb^{-1}.
            \]
            Since every companion matrix is regular, it follows that $\Phi$ takes values in $\nu^{-1}(\gl_{n-1}^{\reg})\subseteq\Wh_{\GL_n}$.
    
            An inverse to $\Phi$ can be constructed as follows. Given $(g,C_p)\in \nu^{-1}(\gl_{n-1}^{\reg})\subseteq\Wh_{\GL_n}$, let $q$ be the characteristic polynomial of $\nu(g,C_p)=\mathrm{pr}(gC_pg^{-1})$. Choose $b\in\GL_{n-1}$ such that $\nu(g,C_p)=bC_qb^{-1}$ and set
            \[
                a:=S_p(\iota(b^{-1})g)^T,\quad \xi:=a^{-1}C_pa.
            \]
            Using $S_p^{-1}C_pS_p=C_p^T$, we obtain
            \[
                \xi= (\iota(b^{-1})gC_pg^{-1}\iota(b))^T,\quad \mathrm{pr}(\xi) = (b^{-1}\mathrm{pr}(gC_pg^{-1})b)^T = C_q^T.
            \]
            Thus $(a,\xi) \in \mu^{-1}(\mathcal{S}_n\times\mathcal{S}_{n-1})$. Since $S_p$ is symmetric, we have
            \[
                \Phi((a,\xi),(b,C_q)) = (\iota(b)a^TS_p^{-1},C_p) = (g,C_p).
            \]
            Moreover, any other choice of $b$ yields $((a,\xi),(b,C_q))$ representing the same class in the quotient. Therefore $\Phi$ is an isomorphism $\Phi\colon \mu^{-1}(\mathcal{S}_n\times\mathcal{S}_{n-1})\times_{\mathfrak{c}_{\GL_{n-1}}}^{J_{\GL_{n-1}}} \Wh_{\GL_{n-1}} \to \nu^{-1}(\gl_{n-1}^{\reg})$.
            
            Since $\Wh_{\GL_n}$ is normal and affine, taking affinization yields the desired isomorphism:
            \[
                \mathrm{Aff}((\mu^{-1}(\mathcal{S}_n \times \mathcal{S}_{n-1}) \times_{\mathfrak{c}_{\GL_{n-1}}}^{{J_{\GL_{n-1}}}} \Wh_{\GL_{n-1}})) \cong \Wh_{\GL_n}.
            \]
            It is an easy consequence of \eqref{eq-jaction1edge1} and the definition of $\Phi$ that the above isomorphism is $J_{\GL_n}$-equivariant. 
        \end{proof}
            
        \begin{proof}[Proof of \Cref{prop-maincoulombbranch}]
            The proof is by induction, the case $n=1$ being trivial and the case $n=2$ being exactly the statement of \Cref{lemma-coulombbranchmain}, since $\Wh_{\GL_1} = J_{\GL_1}$. For $n>2$, assuming the result holds for $n-1$, one must just apply \Cref{lemma-coulombbranchmain} in conjunction with the observations made before its statement regarding quiver dismemberment.
        \end{proof}
    
        \subsection{The regular sheaf}

        Let $G = \prod_{i=1}^{n-1} \GL_i$ and $\mathcal{G}_n = \GL_n \times G$. Let $\pi\colon \mathrm{Gr}_{\mathcal{G}_n} \to \mathrm{Gr}_{\GL_n}$ and let $\mathcal{A}:=\mathcal{A}_{\mathcal{G}_n,\mathcal{V}_n}$ denote the ring object associated to the representation of $\mathcal{G}_n$ afforded by the Gelfand-Tsetlin quiver. Let $\mathcal{R}_{\GL_n}$ denote the regular perverse sheaf on the affine Grassmannian $\mathrm{Gr}_{\GL_n}$. We now use \Cref{prop-maincoulombbranch} and results of \cite{ginzburg2025pointwisepurityderivedsatake} to give a new proof of the following result, which is essentially equivalent to \cite[Theorem 2.11]{BFN-Ring} (see \cite[proof of Lemma 3.5]{gannon2025functorialitycoulombbranches}).

        \begin{thm}
            There is an isomorphism $\mathcal{R}_{\GL_n} \cong \pi_*\mathcal{A}$ of ring objects in the ind-completed derived Satake category $D_{\GL_n(O)}(\mathrm{Gr}_{\GL_n})$.
        \end{thm}
        \begin{proof}
            We have
            \[
                \spec(H^\bullet_{\GL_n(O)}(\mathrm{Gr}_{\GL_n}, \pi_*\mathcal{A})) \cong \spec(H^\bullet_{\mathcal{G}_n(\mathcal{O})}(\mathrm{Gr}_{\mathcal{G}_n},\mathcal{A})).
            \]
            By \Cref{prop-maincoulombbranch}, the right-hand side is $J_{\GL_n}$-equivariantly isomorphic to $\Wh_{\GL_n}$. Let $\Phi$ denote the derived Satake equivalence, normalized as in \cite[5(vi)]{BFN-Ring}. The object $\pi_*\mathcal{A}$ is $!$-pure (this follows from equivariant base change and the fact that $\mathcal{A}$ is $!$-pure, together with \cite[Lemma 5.2.6]{ginzburg2025pointwisepurityderivedsatake}). Therefore, it follows from results of \cite{ginzburg2025pointwisepurityderivedsatake} that
            \begin{align*}
                \spec H^\bullet\Phi(\pi_*\mathcal{A}) &\cong \mathrm{Aff}(\Wh_{\GL_n} \times_{\mathfrak{c}_{\GL_n}}^{J_{\GL_n}} \spec(H^\bullet_{\GL_n(\mathcal{O})}(\pi_*\mathcal{A})))
            \end{align*}
            By \Cref{prop-maincoulombbranch}, $\spec(H^\bullet_{\GL_n(\mathcal{O})}(\pi_*\mathcal{A})) \cong \Wh_{\GL_n}$, and by the Moore-Tachikawa isomorphism (cf.\ \cite[Section 5]{BFN-Ring}) $\mathrm{Aff}(\Wh_{\GL_n} \times_{\mathfrak{c}_{\GL_n}}^{J_{\GL_n}} \Wh_{\GL_n}) \cong T^*\GL_n = \spec H^\bullet\Phi(\mathcal{R}_{\GL_n})$. Therefore we have $\spec H^\bullet\Phi(\pi_*\mathcal{A}) \cong \spec H^\bullet\Phi(\mathcal{R}_{\GL_n})$. This identification is compatible with the $\GL_n$-action, the cohomological grading, and the morphisms to $\gl_n^*$, so we have an isomorphism $H^\bullet\Phi(\pi_*\mathcal{A}) \cong H^\bullet\Phi(\mathcal{R}_{\GL_n})$ of graded $\GL_n$-equivariant $\mathrm{Sym}(\gl_n[-2])$-algebras. Since $H^\bullet\Phi(\mathcal{R}_{\GL_n}) \cong \mathrm{Sym}(\gl_n[-2]) \otimes \mathcal{O}(\GL_n)$ is free over $\mathrm{Sym}(\gl_n[-2])$ and generated in degree zero, this isomorphism lifts to an isomorphism of ring objects in the equivariant derived category $\Phi(\pi_*\mathcal{A}) \cong \Phi(\mathcal{R}_{\GL_n})$.
        \end{proof}
        
    \section{Gluing via involution}\label{section-involution}

        In this section, following a suggestion of Ginzburg, we modify the gluing procedure outlined in \Cref{section-prelim} to make the rational automorphism used as a transition function an involution. As a result, for any complex reductive group $G$ and any gluable finite-dimensional representation $\mathbf{N}$ of $G$, the relative Langlands dual of $T^*\mathbf{N}$ is naturally equipped with an anti-symplectic involution. We give an explicit formula for this involution in the Rankin-Selberg case. I am grateful to Victor Ginzburg for suggesting both the idea and the approach used here.
        
        \subsection{Duality involution}

        Let $G$ be an arbitrary complex reductive group, and $G^\vee$ its Langlands dual. The group $G^\vee$ has a canonical pinning, which determines a maximal unipotent subgroup $U^\vee \subseteq G^\vee$ and a non-degenerate character $\psi\colon \check{\mathfrak{u}} = \mathrm{Lie}(U^\vee) \to \C$.

        Let $\iota\colon G^\vee \to G^\vee$ be the unique involution of $G^\vee$ that negates the pinning and acts on the maximal torus $T^\vee$ by $t \mapsto w_0(t^{-1})$, where $w_0$ denotes the longest element of the Weyl group. Following \cite[Section 2.3]{benzvi2024relativelanglandsduality} and \cite{Prasad2018}, we refer to this as the \emph{duality involution}. Let $I$ denote the \emph{anti-}symplectic lift of $\iota$ to the cotangent bundle $T^*G^\vee$, i.e.\ the composition of the canonical cotangent lift with fiberwise negation. Under the left-invariant trivialization $T^*G^\vee \cong G^\vee \times \check{\mathfrak{g}}^*$, this reads as
        \[
            I\colon (g,\xi) \mapsto (\iota(g),-(\mathrm{d}\iota)^T(\xi)),
        \]
        where $\mathrm{d}\iota\colon \check{\mathfrak{g}} \to \check{\mathfrak{g}}$ is the involution on $\check{\mathfrak{g}}$ and $(-)^T$ denotes transposition of linear maps. 
        
        The moment map of the left $G^\vee$-action on $G^\vee \times \check{\mathfrak{g}}^* \cong T^*G^\vee$ is $\mu\colon (g, \xi) \mapsto \mathrm{Ad}_g(\xi)$. Hence $\mu \circ I = - (\mathrm{d}\iota)^T \circ \mu$, and likewise for the right-action. Since $\iota$ negates the pinning, this implies that $I$ preserves $G^\vee \times (\psi + \check{\mathfrak{u}}^\perp)$, and so $I$ descends to a well-defined anti-symplectic involution $\theta$ on the Whittaker-twisted cotangent bundle $\mathrm{Wh}_{G^\vee} := T^*_{\psi}(G^\vee/U^\vee) = G^\vee \times_{U^\vee} (\psi + \check{\mathfrak{u}}^\perp)$. Moreover, it is clear that $\theta(g \cdot x) = \iota(g) \cdot \theta(x)$ for every $g \in G^\vee$ and $x \in \mathrm{Wh}_{G^\vee}$.
        
            \subsubsection{Example: \texorpdfstring{$\GL_n$}{GLn}}
            In the case of $\GL_n$ with the standard pinning, the duality involution is given explicitly by $g \mapsto J(g^T)^{-1}J$, where $J$ is the antidiagonal matrix
            \begin{equation}\label{eq-matrixantidiagonal}
                J = \begin{pmatrix*}
                    0 & 0 & \cdots & 0 & 1 \\
                    0 &  0 & \cdots & 1 & 0 \\
                    \vdots & & \iddots & 0 & \vdots\\
                    0 & 1 & & \\
                    1 & 0 & \dots & 0 & 0
                \end{pmatrix*}
            \end{equation}
            The anti-symplectic lift $I$ to $T^*\GL_n$, identified via left-invariant trivialization with $\GL_n \times \gl_n$, is given by
            \[
                (g,a) \mapsto (J(g^T)^{-1}J, Ja^TJ).
            \]
            As explained above, this descends to an anti-symplectic involution $\theta\colon \mathrm{Wh}_{\GL_n} \to \mathrm{Wh}_{\GL_n}$. Under the identification $\mathrm{Wh}_{\GL_n} \cong \GL_n \times \mathcal{S}_n$, where $\mathcal{S}_n$ is the slice of companion matrices, this involution is described as follows. 
            
            Recall that if $C_p$ denotes the companion matrix of $p(t) = t^n + p_{n-1}t^{n-1} + \dots + p_0$ then the Bezoutian matrix $S_p$ defined in \eqref{eq-Sp-matrix} satisfies $S_pC_p^T = C_pS_p$. Moreover, the matrix $S_pJ$ is a unipotent upper-triangular matrix, given by
            \[
                S_pJ =
                \begin{pmatrix}
                    1 & p_{n-1} & \dots & p_2 & p_1 \\
                    0 & 1 & \dots & p_3 & p_2 \\
                    \vdots & \vdots & \ddots & \vdots & \vdots \\
                    0 & 0 & \dots & 1 & p_{n-1} \\
                    0 & 0 & \dots & 0 & 1
                \end{pmatrix}
            \]
            Therefore, in the quotient $\mathrm{Wh}_{\GL_n} = \GL_n \times_U (\psi + \mathfrak{u}^\perp)$, $I(g,C_p)$ is identified with its image under the right action of $(S_pJ)^{-1} \in U$. Explicitly, this image is
            \[
                (J(g^T)^{-1}JJS_p^{-1}, S_pJJC_p^TJJS_p^{-1}) = (J(g^T)^{-1}S_p^{-1}, S_pC_p^TS_p^{-1}) = (J(g^T)^{-1}S_p^{-1}, C_p).
            \]
            Since the above lies in $\GL_n \times \mathcal{S}_n$, it follows that the involution $\theta\colon \GL_n \times \mathcal{S}_n \to \GL_n \times \mathcal{S}_n$ is given by
            \begin{equation}\label{eq-involution-GLn}
                \theta(g,C_p) = (J(g^T)^{-1}S_p^{-1},C_p).
            \end{equation}

        \subsection{Modified gluing}

        We keep the notation of \Cref{subsection-gluingoutline}, assuming $\mathbf{N}$ is gluable. We define a rational map $\tau_\rho\colon \mathrm{Wh}_{G^\vee} \dashrightarrow \mathrm{Wh}_{G^\vee}$ by $\tau_\rho := \theta \circ \sigma_\rho$ and let
        \[
            \mathcal{A}_\tau := \{f \in \mathcal{O}(\mathrm{Wh}_{G^\vee}) \mid \tau_{\rho}^*(f) \in \mathcal{O}(\mathrm{Wh}_{G^\vee})\}.
        \]       
        The map $f \mapsto \theta^*(f)$ gives an isomorphism of rings $\mathcal{A}_\tau \to \mathcal{A}_\sigma$, and thus induces an isomorphism $\spec(\mathcal{A}_{\sigma}) \to \spec(\mathcal{A}_\tau)$. Since $\theta(g \cdot x) = \iota(g) \cdot \theta(x)$, this isomorphism twists the $G^\vee$-action by the duality involution $\iota$. Thus $\spec(\mathcal{A}_\tau)$ is isomorphic to the relative Langlands dual of $T^*\mathbf{N} = \mathbf{N} \oplus\mathbf{N}^*$, but with the action twisted by the duality involution $\iota$, corresponding to swapping the components of the polarization of $T^*\mathbf{N}$, i.e.\ $T^*\mathbf{N} = \mathbf{N}^* \oplus (\mathbf{N}^*)^*$.
        
        \begin{prop}\label{prop-involution}
            The map $\tau_\rho$ is an involution.
        \end{prop}
        \begin{proof}
            Write $[g,\xi]$ for the equivalence class of $(g,\xi) \in G^\vee \times (\psi + \check{\mathfrak{u}}^{\perp})$ in the Whittaker-twisted cotangent bundle $\mathrm{Wh}_{G^\vee} := (G^\vee \times (\psi + \check{\mathfrak{u}}^{\perp}))/U^\vee$. The moment map $\mu$ is $\mu([g,\xi]) = \Ad_{g}(\xi)$, and from the definition of $\gamma_\rho$ it follows that $\gamma_\rho(\Ad_g(\xi)) = g\gamma_\rho(\xi)g^{-1}$. Hence,
            \[
                \sigma_{\rho}([g,\xi]) = [\gamma_\rho(\Ad_g(\xi))g,\xi] = [g\gamma_\rho(\xi),\xi].
            \]
            Applying $\theta$, we obtain
            \[
               \tau_{\rho}([g,\xi]) = [\iota(g\gamma_\rho(\xi)), -(\mathrm{d}\iota)^T(\xi)].
            \]
            From this it follows that
            \[
                \tau_{\rho}^2([g,\xi]) = [g\gamma_\rho(\xi) \iota(\gamma_\rho(-(\mathrm{d}\iota)^T(\xi))), \xi].
            \]
            So, $\tau_\rho$ is an involution if and only if $\iota(\gamma_\rho(-(\mathrm{d}\iota)^T(\xi))) = \gamma_\rho(\xi)^{-1}$. It suffices to check this for $\xi \in \check{\mathfrak{t}}^*$, where $\iota$ is given by $t \mapsto w_0(t^{-1})$, $(\mathrm{d}\iota)^T$ is given by $\xi \mapsto -w_0 \cdot \xi$, and $\gamma_\rho$ coincides with the $W$-equivariant map \eqref{eq-telemansection}. Thus
            \[
                \iota(\gamma_\rho(-(\mathrm{d}\iota)^T(\xi))) = (w_0 \cdot \gamma_\rho(w_0 \cdot \xi))^{-1} = (w_0^2 \cdot \gamma_\rho(\xi))^{-1},
            \]
            and the result follows from the fact that $w_0^2=1$.
        \end{proof}

        \begin{rmk}
            By definition, the affine scheme $\spec(\mathcal{A}_\tau)$ comes equipped with two natural maps $j_-,j_+\colon \mathrm{Wh}_{G^\vee} \to \spec(\mathcal{A}_\tau)$, induced from the ring maps $\mathcal{A}_\tau \to \mathcal{O}(\mathrm{Wh}_{G^\vee})$. Moreover, $\tau_\rho$ induces an involution $\tau_\rho$ of $\mathcal{A}_\tau$ such that $j_+ = j_- \circ \tau_\rho$.
        \end{rmk}

        \subsubsection{Example: Rankin-Selberg}

            We now explicitly compute the involution $\tau_\rho$ on the Hamiltonian space obtained by gluing in the Rankin-Selberg case. Recall the notation from \Cref{section-mainthm}: $G = G^\vee = \GL_n \times \GL_n$, and $\rho$ is the representation of $G$ on $\Hom(\C^n,\C^n)$ by $(g,h) \cdot a = gah^{-1}$. We use the identification $\mathrm{Wh}_{G^\vee} = \GL_n \times \GL_n \times \mathcal{S}_n \times \mathcal{S}_n$. By \eqref{eq-involution-GLn} above, the duality involution on $\mathrm{Wh}_{G^\vee}$ is given by
            \[
                (g,h,C_p,C_q) \mapsto (J(g^T)^{-1}S_p^{-1}, J(h^T)^{-1}S_q^{-1},C_p,C_q),
            \]
            where $C_p$ denotes the companion matrix of $p$ and $S_p$ is the Bezoutian matrix \eqref{eq-Sp-matrix}, and similarly for $C_q$, $S_q$.

            Recall formula \eqref{eq-formula-sigma} for the birational automorphism $\sigma_\rho$: 
            \[
                \sigma_{\rho}(g,h,C_p,C_q) = (gq(C_p), (-1)^nhp(C_q)^{-1}, C_p, C_q).
            \]
            Thus, $\tau_\rho := \theta \circ \sigma_\rho$ is given by
            \begin{align*}
                \tau_{\rho}(g,h,C_p,C_q) &= (J(g^T)^{-1}q(C_p^T)^{-1}S_p^{-1}, (-1)^nJ(h^T)^{-1}p(C_q^T)S_q^{-1}, C_p, C_q).
            \end{align*}

            Write $T^*(\GL_n \times \C^n) = \GL_n \times \gl_n \times \C^n \times (\C^n)^*$ and recall the morphisms $\Phi_1$ and $\Phi_2$ from \eqref{eq-morphism1-RS} and \eqref{eq-morphism2-RS}:
            \begin{align*}
                \Phi_1(g,h,C_p,C_q) &= (gS_ph^T,(h^T)^{-1}C_q^Th^T,ge_1,e_n^Tq(C_p)g^{-1}) \\
                \Phi_2(g,h,C_p,C_q) &= ((-1)^{n-1}gS_q h^T,(h^T)^{-1}C_q^Th^T,gq(C_p)e_1,e_n^Tg^{-1}). 
            \end{align*}
            Recall that $\Phi_2 = \Phi_1 \circ \sigma_\rho$. Let $\Psi_1 := \Phi_1 \circ \theta$.  Explicitly,
            \begin{equation*}
                \Psi_1(g,h,C_p,C_q) = (J(g^T)^{-1}S_q^{-1}h^{-1}J, JhC_qh^{-1}J, J(g^T)^{-1}e_n, e_1^Tq(C_p^T)g^TJ).
            \end{equation*}
            Since $\theta$ is an involution, we have 
            \[
                \Psi_1 \circ \tau_{\rho} = \Phi_1 \circ \sigma_{\rho} = \Phi_2.
            \]
            Therefore, using the same argument as in the proof of \Cref{thm-main2}, but with $\Psi_1$ in place of $\Phi_1$, we obtain:
            
            \begin{corollary}
                There is an isomorphism of affine schemes
                \[
                    \spec(\mathcal{A}_\tau) \cong T^*(\GL_n \times \C^n) = \GL_n \times \gl_n \times \C^n \times (\C^n)^*
                \]
                which intertwines the $G^\vee$-action on $\spec(\mathcal{A}_\tau)$ with the action 
                \begin{equation}\label{eq-actionnew}
                    (g,h) \cdot (a,y,v,u) = (J(g^T)^{-1}JaJh^{-1}J, JhJyJh^{-1}J, J(g^T)^{-1}Jv, uJg^TJ).
                \end{equation}                                
                Moreover, the involution $\tau_{\rho}$ on $\spec(\mathcal{A}_\tau)$ is intertwined by this isomorphism with the involution on $T^*(\GL_n \times \C^n)$ given by
                \begin{equation}\label{eq-involutionRS}
                    (a,y,v,u) \mapsto ((-1)^{n-1}J(a^T)^{-1}J, Jy^TJ, Ju^T, v^TJ).
                \end{equation}
            \end{corollary}

        \begin{rmk}
            The action given by \eqref{eq-actionnew} is the action defined in \Cref{corollary-RS}, but twisted by the duality involution $\iota(g,h) = (J(g^T)^{-1}J,J(h^T)^{-1}J)$.
            
            Moreover, the involution given by \eqref{eq-involutionRS} is equal to the composition $\tau_0 \circ \alpha$, where $\tau_0$ is the anti-symplectic involution on $T^*(\GL_n \times \C^n)$ given by
            \[
                \tau_0(a,y,v,u) = ((-1)^{n-1}(a^T)^{-1},y^T,u^T,v^T)
            \]
            and $\alpha$ is given by the action of $(J,J)$, i.e.\ $\alpha(a,y,v,u) = (JaJ,JyJ,Jv,uJ)$.
        \end{rmk}

        \begin{rmk}
            In the case of the \emph{universal Coulomb branch} considered in \cite[Section 3.3]{ginzburg2025pointwisepurityderivedsatake}, where $\mathbf{N} = \C^n$ is the standard representation of $\GL_n$, the cotangent bundle $T^*\gl_n = \gl_n \times \gl_n$ is realized as the gluing of two copies of $T^*\GL_n$. The corresponding anti-symplectic involution of $T^*\gl_n$ is then $(a,b) \mapsto (b^T,a^T)$. 
        \end{rmk}
        
        \printbibliography

@misc{ginzburg2025pointwisepurityderivedsatake,
      title={Pointwise purity, derived Satake, and Symplectic duality}, 
      author={Victor Ginzburg},
      year={2025},
      eprint={2508.15958},
      archivePrefix={arXiv},
      primaryClass={math.RT},
      url={https://arxiv.org/abs/2508.15958}, 
}

@misc{gannon2025functorialitycoulombbranches,
      title={Functoriality of Coulomb branches}, 
      author={Tom Gannon and Ben Webster},
      year={2025},
      eprint={2501.09962},
      archivePrefix={arXiv},
      primaryClass={math.AG},
      url={https://arxiv.org/abs/2501.09962}, 
}

@Article{Teleman2021,
  author    = {Teleman, Constantin},
  journal   = {Journal of the European Mathematical Society},
  title     = {The rôle of Coulomb branches in 2D gauge theory},
  year      = {2021},
  issn      = {1435-9863},
  month     = Jun,
  number    = {11},
  pages     = {3497--3520},
  volume    = {23},
  doi       = {10.4171/jems/1071},
  publisher = {European Mathematical Society - EMS - Publishing House GmbH},
}

@article{HELMKE19891039,
    title = {Bezoutians},
    journal = {Linear Algebra and its Applications},
    volume = {122-124},
    pages = {1039-1097},
    year = {1989},
    note = {Special Issue on Linear Systems and Control},
    issn = {0024-3795},
    doi = {https://doi.org/10.1016/0024-3795(89)90684-8},
    url = {https://www.sciencedirect.com/science/article/pii/0024379589906848},
    author = {U. Helmke and P.A. Fuhrmann}
    }

@misc{benzvi2024relativelanglandsduality,
      title={Relative Langlands Duality}, 
      author={David Ben-Zvi and Yiannis Sakellaridis and Akshay Venkatesh},
      year={2024},
      eprint={2409.04677},
      archivePrefix={arXiv},
      primaryClass={math.RT},
      url={https://arxiv.org/abs/2409.04677}, 
}

@Article{FengWang2025,
    author    = {Feng, Tony and Wang, Jonathan},
    journal   = {Geometric and Functional Analysis},
    title     = {Geometric Langlands Duality for Periods},
    year      = {2025},
    issn      = {1420-8970},
    month     = Feb,
    number    = {2},
    pages     = {463--541},
    volume    = {35},
    doi       = {10.1007/s00039-025-00702-4},
    publisher = {Springer Science and Business Media LLC},
}

@misc{nakajima2024sdualhamiltonianmathbfg,
      title={S-dual of Hamiltonian $\mathbf G$ spaces and relative Langlands duality}, 
      author={Hiraku Nakajima},
      year={2024},
      eprint={2409.06303},
      archivePrefix={arXiv},
      primaryClass={math.AG},
      url={https://arxiv.org/abs/2409.06303}, 
}

@Article{Takayama2016,
  author    = {Takayama, Yuuya},
  journal   = {Publications of the Research Institute for Mathematical Sciences},
  title     = {Nahm’s Equations, Quiver Varieties and Parabolic Sheaves},
  year      = {2016},
  issn      = {1663-4926},
  month     = Jan,
  number    = {1},
  pages     = {1--41},
  volume    = {52},
  doi       = {10.4171/prims/172},
  publisher = {European Mathematical Society - EMS - Publishing House GmbH},
}

@Article{NakajimaTakayama2017,
  author    = {Nakajima, Hiraku and Takayama, Yuuya},
  journal   = {Selecta Mathematica},
  title     = {Cherkis bow varieties and Coulomb branches of quiver gauge theories of affine type A},
  year      = {2017},
  issn      = {1420-9020},
  month     = Jun,
  number    = {4},
  pages     = {2553--2633},
  volume    = {23},
  doi       = {10.1007/s00029-017-0341-7},
  publisher = {Springer Science and Business Media LLC},
}

@article{Finkelberg_2025,
   title={Hyperspherical Equivariant Slices and Basic Classical Lie Superalgebras},
   volume={406},
   ISSN={1432-0916},
   url={http://dx.doi.org/10.1007/s00220-025-05378-3},
   DOI={10.1007/s00220-025-05378-3},
   number={8},
   journal={Communications in Mathematical Physics},
   publisher={Springer Science and Business Media LLC},
   author={Finkelberg, Michael and Ukraintsev, Ivan},
   year={2025},
   month=Jul }

@article{Bielawski1997,
author = {Bielawski, Roger},
title = {Hyperkähler Structures and Group Actions},
journal = {Journal of the London Mathematical Society},
volume = {55},
number = {2},
pages = {400-414},
doi = {10.1112/s0024610796004723},
url = {https://londmathsoc.onlinelibrary.wiley.com/doi/abs/10.1112/S0024610796004723},
year = {1997}
}

@article{Sylvester1853,
    author = {Sylvester, James Joseph},
    title = {XVIII. On a theory of the syzygetic relations of two rational integral functions, comprising an application to the theory of Sturm’s functions, and that of the greatest algebraical common measure},
    journal = {Philosophical Transactions of the Royal Society of London},
    number = {143},
    pages = {407-548},
    year = {1853},
    month = {12},
    issn = {0261-0523},
    doi = {10.1098/rstl.1853.0018},
    url = {https://doi.org/10.1098/rstl.1853.0018},
    eprint = {https://royalsocietypublishing.org/rstl/article-pdf/doi/10.1098/rstl.1853.0018/1454780/rstl.1853.0018.pdf},
}

@article{HeinigRost,
title = {On the inverses of Toeplitz-plus-Hankel matrices},
journal = {Linear Algebra and its Applications},
volume = {106},
pages = {39-52},
year = {1988},
issn = {0024-3795},
doi = {https://doi.org/10.1016/0024-3795(88)90021-3},
url = {https://www.sciencedirect.com/science/article/pii/0024379588900213},
author = {Georg Heinig and Karla Rost}
}

@Article{Hurtubise1989,
  author    = {Hurtubise, Jacques},
  journal   = {Communications In Mathematical Physics},
  title     = {The classification of monopoles for the classical groups},
  year      = {1989},
  issn      = {1432-0916},
  month     = Dec,
  number    = {4},
  pages     = {613--641},
  volume    = {120},
  doi       = {10.1007/bf01260389},
  publisher = {Springer Science and Business Media LLC},
}

@Article{BFN-Ring,
  author    = {Braverman, Alexander and Finkelberg, Michael and Nakajima, Hiraku},
  journal   = {Advances in Theoretical and Mathematical Physics},
  title     = {Ring objects in the equivariant derived Satake category arising from Coulomb branches},
  year      = {2019},
  issn      = {1095-0753},
  number    = {2},
  pages     = {253--344},
  volume    = {23},
  doi       = {10.4310/atmp.2019.v23.n2.a1},
  publisher = {International Press of Boston},
}

@Article{Lusztig1976,
  author    = {Lusztig, G.},
  journal   = {Inventiones Mathematicae},
  title     = {Coxeter orbits and eigenspaces of Frobenius},
  year      = {1976},
  issn      = {1432-1297},
  month     = Jun,
  number    = {2},
  pages     = {101--159},
  volume    = {38},
  doi       = {10.1007/bf01408569},
  publisher = {Springer Science and Business Media LLC},
}

@Article{Knop1996,
  author    = {Knop, Friedrich},
  journal   = {Journal of the American Mathematical Society},
  title     = {Automorphisms, root systems, and compactifications of homogeneous varieties},
  year      = {1996},
  issn      = {0894-0347},
  number    = {1},
  pages     = {153--174},
  volume    = {9},
  doi       = {10.1090/s0894-0347-96-00179-8},
  publisher = {American Mathematical Society (AMS)},
}

@article {NgoLeLemmeFondamentalPourLesAlgebresdeLie,
    AUTHOR = {Ng\^{o}, Bao Ch\^{a}u},
     TITLE = {Le lemme fondamental pour les alg\`ebres de {L}ie},
   JOURNAL = {Publ. Math. Inst. Hautes \'{E}tudes Sci.},
  FJOURNAL = {Publications Math\'{e}matiques. Institut de Hautes \'{E}tudes
              Scientifiques},
    NUMBER = {111},
      YEAR = {2010},
     PAGES = {1--169},
      ISSN = {0073-8301,1618-1913},
   MRCLASS = {22E35 (11S37 14D23 14G35 22E50)},
  MRNUMBER = {2653248},
MRREVIEWER = {R.\ P.\ Langlands},
       DOI = {10.1007/s10240-010-0026-7},
       URL = {https://doi.org/10.1007/s10240-010-0026-7},
}

@article {DonagiGaitsgoryTheGerbeofHiggsBundles,
    AUTHOR = {Donagi, R. Y. and Gaitsgory, D.},
     TITLE = {The gerbe of {H}iggs bundles},
   JOURNAL = {Transform. Groups},
  FJOURNAL = {Transformation Groups},
    VOLUME = {7},
      YEAR = {2002},
    NUMBER = {2},
     PAGES = {109--153},
      ISSN = {1083-4362,1531-586X},
   MRCLASS = {14J60 (14D21 14L30)},
  MRNUMBER = {1903115},
MRREVIEWER = {Adrian\ Langer},
       DOI = {10.1007/s00031-002-0008-z},
       URL = {https://doi.org/10.1007/s00031-002-0008-z},
}

@Article{Kac1980,
  author    = {Kac, V.G.},
  journal   = {Journal of Algebra},
  title     = {Some remarks on nilpotent orbits},
  year      = {1980},
  issn      = {0021-8693},
  month     = May,
  number    = {1},
  pages     = {190--213},
  volume    = {64},
  doi       = {10.1016/0021-8693(80)90141-6},
  publisher = {Elsevier BV},
}

@Article{Vinberg1978,
  author    = {Vinberg, É. B. and Kimel’fel’d, B. N.},
  journal   = {Functional Analysis and Its Applications},
  title     = {Homogeneous domains on flag manifolds and spherical subgroups of semisimple Lie groups},
  year      = {1978},
  issn      = {1573-8485},
  month     = Jul,
  number    = {3},
  pages     = {168--174},
  volume    = {12},
  doi       = {10.1007/bf01681428},
  publisher = {Springer Science and Business Media LLC},
}

@Article{Vinberg2001,
  author    = {Vinberg, E B},
  journal   = {Russian Mathematical Surveys},
  title     = {Commutative homogeneous spaces and co-isotropic symplectic actions},
  year      = {2001},
  issn      = {1468-4829},
  month     = Feb,
  number    = {1},
  pages     = {1--60},
  volume    = {56},
  doi       = {10.1070/rm2001v056n01abeh000356},
  publisher = {Steklov Mathematical Institute},
}

@Article{BFM,
  author    = {Bezrukavnikov, Roman and Finkelberg, Michael and Mirkovic, Ivan},
  journal   = {Compositio Mathematica},
  title     = {Equivariant homology and K-theory of affine Grassmannians and Toda lattices},
  year      = {2005},
  issn      = {1570-5846},
  month     = Apr,
  number    = {03},
  pages     = {746--768},
  volume    = {141},
  doi       = {10.1112/s0010437x04001228},
  publisher = {Cambridge University Press (CUP)},
}

@Article{Prasad2018,
  author    = {Prasad, Dipendra},
  journal   = {Transactions of the American Mathematical Society},
  title     = {Generalizing the MVW involution, and the contragredient},
  year      = {2018},
  issn      = {1088-6850},
  month     = Nov,
  number    = {1},
  pages     = {615--633},
  volume    = {372},
  doi       = {10.1090/tran/7602},
  publisher = {American Mathematical Society (AMS)},
}

@Article{MWZ26,
  author    = {Mao, Zhengyu and Wan, Chen and Zhang, Lei},
  journal   = {Forum of Mathematics, Sigma},
  title     = {Strongly tempered hyperspherical Hamiltonian spaces},
  year      = {2026},
  issn      = {2050-5094},
  volume    = {14},
  doi       = {10.1017/fms.2026.10190},
  publisher = {Cambridge University Press (CUP)},
}

@Article{BFN2018,
  author    = {Braverman, Alexander and Finkelberg, Michael and Nakajima, Hiraku},
  journal   = {Advances in Theoretical and Mathematical Physics},
  title     = {Towards a mathematical definition of Coulomb branches of $3$-dimensional $\mathcal{N} = 4$ gauge theories, II},
  year      = {2018},
  issn      = {1095-0753},
  number    = {5},
  pages     = {1071--1147},
  volume    = {22},
  doi       = {10.4310/atmp.2018.v22.n5.a1},
  publisher = {International Press of Boston},
}

@InBook{BF2019,
  author    = {Braverman, Alexander and Finkelberg, Michael},
  pages     = {1--52},
  publisher = {Springer International Publishing},
  title     = {Coulomb Branches of 3-Dimensional Gauge Theories and Related Structures},
  year      = {2019},
  isbn      = {9783030268565},
  booktitle = {Geometric Representation Theory and Gauge Theory},
  doi       = {10.1007/978-3-030-26856-5_1},
  issn      = {1617-9692},
}

@misc{bellamy2023coulombbranchessymplecticsingularities,
      title={Coulomb Branches have symplectic singularities}, 
      author={Gwyn Bellamy},
      year={2023},
      eprint={2304.09213},
      archivePrefix={arXiv},
      primaryClass={math.AG},
      url={https://arxiv.org/abs/2304.09213}, 
}

@Article{GannonWilliams2026,
  author    = {Gannon, Tom and Williams, Harold},
  journal   = {International Mathematics Research Notices},
  title     = {Differential Operators on the Base Affine Space of SLn and Quantized Coulomb Branches},
  year      = {2026},
  issn      = {1687-0247},
  month     = Jan,
  number    = {2},
  volume    = {2026},
  doi       = {10.1093/imrn/rnaf384},
  publisher = {Oxford University Press (OUP)},
}

@Article{BezrukavnikovFinkelberg2008,
  author    = {Bezrukavnikov, R. and Finkelberg, M.},
  journal   = {Moscow Mathematical Journal},
  title     = {Equivariant Satake Category and Kostant–Whittaker Reduction},
  year      = {2008},
  issn      = {1609-4514},
  number    = {1},
  pages     = {39--72},
  volume    = {8},
  doi       = {10.17323/1609-4514-2008-8-1-39-72},
  publisher = {National Research University, Higher School of Economics (HSE)},
}

@misc{teleman2026coulombbranchesquaternionicrepresentations,
      title={Coulomb branches for quaternionic representations}, 
      author={Constantin Teleman},
      year={2026},
      eprint={2209.01088},
      archivePrefix={arXiv},
      primaryClass={math.AT},
      url={https://arxiv.org/abs/2209.01088}, 
}
        \nocite{*}
\end{document}